\documentclass[12pt,leqno,a4paper]{amsart}
\usepackage{enumerate,cite}
\usepackage{amssymb}
\usepackage{amsthm}      
\usepackage{amsmath}
\usepackage{txfonts}
\usepackage{mathrsfs}
\usepackage{hyperref}
\usepackage{color}
\usepackage[all]{xy}

\newcommand{\F}{\mathbb{F}}

\newcommand{\G}{\mathrm{G}}
\newcommand{\R}{\mathrm{R}}

\newcommand{\bT}{\mathbf{T}}
\newcommand{\bB}{\mathbf{B}}
\newcommand{\bG}{\mathbf{G}}
\newcommand{\bX}{\mathbf{X}}
\newcommand{\bP}{\mathbf{P}}

\newcommand{\bL}{\mathbf{L}}

\newcommand{\bN}{\mathbf{N}}

\newcommand{\bY}{\mathbf{Y}}
\newcommand{\bU}{\mathbf{U}}
\newcommand{\bBr}{\mathbf{Br}}
\newcommand{\bNBr}{\mathbf{NBr}}

\newcommand{\cN}{\mathcal{N}}
\newcommand{\cB}{\mathcal{B}}
\newcommand{\cA}{\mathcal{A}}

\newcommand{\cF}{\mathcal{F}}

\newcommand{\cO}{\mathcal{O}}
\newcommand{\cP}{\mathcal{P}}
\newcommand{\cD}{\mathcal{D}}
\newcommand{\cC}{\mathcal{C}}

\newcommand{\fS}{\mathfrak{S}}

\newcommand{\mrO}{\mathrm{O}}
\newcommand{\Aut}{\operatorname{Aut}\nolimits}
\newcommand{\End}{\operatorname{End}\nolimits}

\newcommand{\Alp}{\operatorname{Alp}\nolimits}
\newcommand{\IBr}{\operatorname{IBr}\nolimits}
\newcommand{\Ind}{\operatorname{Ind}}
\newcommand{\Lin}{\operatorname{Lin}}
\newcommand{\Irr}{\operatorname{Irr}\nolimits}
\newcommand{\Out}{\operatorname{Out}\nolimits}
\newcommand{\Res}{\operatorname{Res}}
\newcommand{\GL}{\operatorname{GL}}

\newcommand{\SL}{\operatorname{SL}}
\newcommand{\SU}{\operatorname{SU}}

\newcommand{\PSp}{\operatorname{PSp}}
\newcommand{\dz}{\operatorname{dz}}

\newcommand{\PSL}{\operatorname{PSL}}
\newcommand{\PSU}{\operatorname{PSU}}

\newcommand{\Br}{\operatorname{Br}}
\newcommand{\bl}{\operatorname{bl}}
\newcommand{\Char}{\operatorname{Char}}
\newcommand{\BrCh}{\operatorname{BrCh}}
\newcommand{\opp}{\operatorname{opp}}
\newcommand{\N}{\operatorname{N}}
\newcommand{\C}{\operatorname{C}}
\newcommand{\Z}{\operatorname{Z}}
\newcommand{\Tr}{\operatorname{Tr}}

\newcommand{\tpsi}{\widetilde{\psi}}
\newcommand{\hpsi}{\widehat{\psi}}
\newcommand{\tQ}{\widetilde{Q}}

\newcommand{\tvhi}{\widetilde{\vhi}}
\newcommand{\hvhi}{\widehat{\vhi}}

\newcommand{\tG}{\widetilde{G}}
\newcommand{\tN}{\widetilde{N}}
\newcommand{\tM}{\widetilde{M}}
\newcommand{\hG}{\widehat{G}}

\newcommand{\hM}{\widehat{M}}

\newcommand{\tL}{\widetilde{L}}

\newcommand{\tH}{\widetilde{H}}

\newcommand{\tcB}{\widetilde{\cB}}
\newcommand{\hcD}{\widehat{\cD}}
\newcommand{\tcD}{\widetilde{\cD}}
\newcommand{\hcB}{\widehat{\cB}}
\newcommand{\tbG}{\widetilde{\mathbf{G}}}
\newcommand{\tbL}{\widetilde{\mathbf{L}}}
\newcommand{\tbN}{\widetilde{\mathbf{N}}}

\newcommand{\ttheta}{\widetilde{\theta}}
\newcommand{\htheta}{\widehat{\theta}}

\newcommand{\wOm}{\widetilde{\Omega}}
\newcommand{\hOm}{\widehat{\Omega}}

\let\eps=\epsilon
\let\ga=\gamma

\let\la=\lambda
\let\vhi=\varphi
\let\Ga=\Gamma

\let\ti=\times

\theoremstyle{theorem}
\newtheorem{mainthm}{Theorem}
\newtheorem{thm}{Theorem}[section]
\newtheorem{lem}[thm]{Lemma}
\newtheorem{prop}[thm]{Proposition}

\newtheorem*{conj}{Conjecture}

\theoremstyle{definition}

\newtheorem{rmk}[thm]{Remark}

\newtheorem*{rem}{Remark}
\numberwithin{equation}{section}

\begin{document}

\title[Inductive blockwise Alperin weight condition for type $\mathsf A$]{Morita equivalences and the inductive blockwise Alperin weight condition for type $\mathsf A$}

\author{Zhicheng Feng}
\address{School of Mathematics and Physics, University of Science and Technology Beijing, Beijing 100083, China\\
{\rm Current address}: SICM and Department of Mathematics, Southern University of Science and Technology, Shenzhen 518055, China}
\email{fengzc@sustech.edu.cn}

\author{Zhenye Li}
\address{College of Mathematics and Physics, Beijing University of Chemical Technology, Beijing 100029, China}
\email{lizhenye@pku.edu.cn}

\author{Jiping Zhang}
\address{SICM and Department of Mathematics, Southern University of Science and Technology, Shenzhen 518055, China;  School of Mathematical Sciences, Peking University, Beijing 100871, China.}
\email{jzhang@pku.edu.cn}

\thanks{Supported by the NSFC~(No. 12001032, 11901028, 11631001).}

\begin{abstract}
As a step to establish the blockwise Alperin weight conjecture for all finite
groups, we verify the inductive blockwise Alperin weight condition introduced by Navarro--Tiep and Sp\"ath for
simple groups of Lie type $\mathsf A$, split or twisted.
Key to the proofs is to reduce the verification of the inductive condition to the isolated (that means unipotent) blocks,  using the Jordan decomposition for blocks of finite reductive groups given by Bonnaf\'e, Dat and Rouquier.
\end{abstract}




\maketitle


\section{Introduction}

The Alperin weight conjecture is one of the most important, basic local-global conjectures in the representation theory of finite groups.
For  a finite group $G$, we denote by $\dz(G)$ the set of defect zero irreducible characters of $G$.
For a prime $\ell$, an \emph{$\ell$-weight} of $G$ means a pair $(Q,\vhi)$, where $Q$ is an $\ell$-subgroup of $G$ and $\vhi\in\dz(\N_G(Q)/Q)$. 
The group $G$ acts by conjugation on the weights of $G$ by $(Q,\vhi)^g=(Q^g,\vhi^g)$.
Denote by $\Alp(G)$ the $G$-conjugacy classes of the weights of $G$.
Let $B$ be an $\ell$-block of $G$. Then $(Q,\vhi)$ is called a \emph{$B$-weight} if $\bl(\vhi)^G=B$, where $\bl(\vhi)$ denotes the block of $\N_G(Q)$ containing the lift of $\vhi$ to $\N_G(Q)$ and $\bl(\vhi)^G$ is the induced block.
Denote by $\Alp(B)$ the $G$-conjugacy classes of $B$-weights of $G$.
Then the blockwise  Alperin weight conjecture  asserts that 

\begin{conj}[Alperin \cite{Al87}]
	Let $G$ be a finite group, $\ell$ a prime and let $B$ be an $\ell$-block of $G$. Then
$$|\IBr(B)|=|\Alp(B)|.$$
\end{conj}

Even though the conjecture was verified for numerous groups,
the general proof seems extremely difficult.
An accessible way is to reduce it to simple groups.
In 2011, Navarro and Tiep \cite{NT11} reduced the non-blockwise version of Alperin weight conjecture to simple groups, and then Sp\"ath \cite{Sp13} achieved a reduction theorem for the blockwise Alperin weight conjecture. 
In this way, to prove the blockwise  Alperin weight  conjecture, it suffices to show the so-called inductive  blockwise Alperin weight (BAW) condition for all finite non-abelian simple groups.

The reduction gives hope that the Alperin weight conjecture can be proved by use of the classification of the finite simple groups.
In several cases the inductive BAW condition
has been verified:
simple alternating groups (Malle \cite{Ma14}),
sporadic simple groups (Breuer--Feng \cite{BF}),
groups of Lie type in their defining characteristic (Sp\"ath \cite{Sp13}),
Suzuki and Ree groups (Malle \cite{Ma14}),
groups of types $\mathsf G_2$ and $^3\mathsf D_4$ (Schulte \cite{Sch16}),
groups of type $\mathsf C$  at the prime $2$ (Feng--Malle \cite{FM20}), groups of type $\mathsf F_{\!4}$ at odd primes (An--Hi\ss--L\"ubeck \cite{AHL21}).
Under the unitriangularity assumption of the decomposition matrices, the inductive BAW condition was verified for groups of type $\mathsf B$ (Feng--C. Li--Zhang \cite{FLZ21b})  and type $\mathsf C$ (C. Li \cite{Li21}) at odd primes.
See also \cite{FZ22} for more developments in the inductive investigation.

In this paper, we verify the inductive BAW condition for simple groups of type $\mathsf A$.

\begin{mainthm}\label{main-thm}
The simple groups $\PSL_n(q)$ and $\PSU_n(q)$ satisfy the inductive blockwise Alperin weight condition.
\end{mainthm}

Let $G=\SL_n(\eps q)$ ($\eps=\pm1$), where $\SL_n(-q)$ denotes $\SU_n(q)$.
In \cite{FLZ20a}, the weights of $G$ were classified and the inductive condition of the non-blockwise Alperin weight conjecture was verified for groups of type $\mathsf A$.
By the criterion given by Brough--Sp\"ath \cite{BS20},
the key of the verification of the inductive BAW condition for the simple group $G/\Z(G)$ is to  establish an $\Aut(G)$-equivariant bijection between $\IBr(G)$ and $\Alp(G)$ which preserves blocks.

Let $p$ be a prime and $\F_p$ the field of $p$ elements.
Let $\bG$ be a simple, simply connected algebraic group over the algebraic closure $\overline{\F}_p$ of $\F_p$ and let $F:\bG\to\bG$ be a Steinberg endomorphism.
Fix a prime $\ell$ different from $p$ and let $(K,\cO,k)$ be an $\ell$-modular system which is large enough for our groups considered.
Denote $\Lambda=\cO$ or $k$.
Let $(\bG^*,F^*)$ be in duality with $(\bG,F)$. 
For simplicity we use the same letter $F$ for $F^*$.
If $s$ is a semisimple $\ell'$-element in ${\bG^*}^F$, then by a theorem of Brou\'e--Michel \cite{BM89}, there is a unique central idempotent $e_s^{\bG^F}$ of $\Lambda \bG^F$ associated to~$s$.

Assume that $\bL^*$ is the minimal Levi subgroup of $\bG^*$ containing $\C^\circ_{\bG^*}(s)$.
Let $\bL$ be the Levi subgroup of $\bG$ which is dual to $\bL^*$ and $N$ the common stabilizer of $\bL$ and $e_s^{\bL^F}$ in $\bG^F$.
By a theorem of Bonnaf\'e--Dat--Rouquier (see \cite{BDR17} and \cite{Ru20}),
if $N/\bL^F$ is cyclic, then
there exists a splendid Rickard equivalence between
$\Lambda Ne_s^{\bL^F}$ and $\Lambda \bG^F e_s^{\bG^F}$, which gives a Morita equivalence.

When focusing on groups of type $\mathsf A$, we let $\bG=\SL_n(\overline{\F}_p)$ so that $\bG^F=G=\SL_n(\eps q)$.
In \cite{FLZ21}, the authors reduced the inductive BAW condition to the so-called strictly
 quasi-isolated blocks.
Now let $s$ be strictly quasi-isolated, \emph{i.e.}, $\C_{{\bG^*}^F}(s)\C_{\bG^*}^\circ(s)$ is not contained in  a proper Levi subgroup of $\bG^*$.
Using the classification of quasi-isolated elements of $\bG^*$ by Bonnaf\'e \cite{Bo05}, Ruhstorfer \cite{Ru21} constructed an $F$-stable Levi subgroup $\bL'$ of $\bG$ containing $\bL$.
This Levi subgroup $\bL'$ satisfies that $N'/{\bL'}^F$ is cyclic and of prime order, where 
$N'$ is the common stabilizer of $\bL'$ and $e_s^{{\bL'}^F}$ in $\bG^F$.
There exists a splendid Rickard equivalence between
$\Lambda N'e_s^{{\bL'}^F}$ and $\Lambda G e_s^{G}$, which gives a Morita equivalence.
He also discussed the interplay of this Morita equivalence with group automorphisms.

For the local situation, Ruhstorfer \cite{Ru21} gave an equivariant Morita equivalence between certain blocks of $\N_{G}(Q)$ and $\N_{N'}(Q)$.
Using the similar methods of \cite{FLZ21}, we show that 
this also induces a Morita equivalence between the corresponding blocks of the quotient groups
$\N_{G}(Q)/Q$ and $\N_{N'}(Q)/Q$.

In this way, we get an equivariant bijection between $\IBr(N',e_s^{{\bL'}^F})$ and $\IBr(G,e_s^{G})$ and an equivariant bijection between $\Alp(N',e_s^{{\bL'}^F})$ and $\Alp(G,e_s^{G})$.
Thus, to give an equivariant bijection between $\IBr(G,e_s^{G})$ and $\Alp(G,e_s^{G})$, it suffices to establish such bijection between $\IBr(N',e_s^{{\bL'}^F})$ and $\Alp(N',e_s^{{\bL'}^F})$.
We improve the methods in the proof of \cite[Thm.~1]{FLZ21}, and reduce the verification of the inductive BAW condition in type $\mathsf A$ to isolated blocks, which was verified in \cite{FLZ20b}.

\vspace{1.5ex}

This paper is structured in the following way.
In Section~\ref{sec:preli}, we introduce the general notation in the representation theory of finite groups.
Section~\ref{sec:char-triple}  considers the bijection between irreducible Brauer characters and weights in terms of normal subgroups and modular character triples.
In Section~\ref{sec:Jor-wei}, we establish an equivariant bijection of irreducible Brauer characters (and weights) between certain blocks of $N'$ and $G$.
Section~\ref{sec:red-iso}  reduces the verification of the inductive BAW condition for groups of type $\mathsf A$ to the isolated blocks and
completes the proof of Theorem~\ref{main-thm}.

\vskip 1pc
\noindent{\bf Acknowledgement:} 
We thank Lucas Ruhstorfer for discussions on an earlier version, and thank Gabriel Navarro for his comment.
We are also indebted to the anonymous referee for many helpful remarks and suggestions that have helped improve this paper.

\section{Preliminaries}\label{sec:preli}

In this section, we recollect some notation and concepts related to the representation theory of finite groups, which mainly follow those used in \cite{BDR17}, \cite{Ru20a} and \cite{FLZ21}.
Throughout, we fix a prime number $\ell$ and consider modular representations
with respect to $\ell$.

\subsection{Notation}

Let $G$ be a finite group.
For the notation on the (Brauer) characters and blocks of $G$, we mainly follow \cite{NT89} and \cite{Na98}.
We denote by $\Res$ and $\Ind$ the restriction and induction respectively.
For a normal subgroup $N$ of $G$, we often identify the (Brauer) characters of $G/N$ with their lifts to $G$.
If $\chi\in\Irr(G)\cup\IBr(G)$, we denote by $\bl_G(\chi)$ the block of $G$ containing $\chi$.
We also write it briefly by $\bl(\chi)$ sometimes.
For $\chi\in\Irr(G)$, we denote by $\chi^\circ$ the restriction of $\chi$ to the $\ell'$-elements of $G$.

We denote by $\Alp^0(G)$ the set of weights of $G$.
Let $(Q,\vhi)\in\Alp^0(G)$.
Then $\vhi\in\dz(\N_G(Q)/Q)$ and thus
 $Q$ is necessarily a radical subgroup of $G$, \emph{i.e.},
$Q=\mrO_\ell(\N_G(Q))$.
As usual, we also regard $\vhi$ as a character of $\N_G(Q)$ and call it a \emph{weight character}.
If $(Q,\vhi)$ is a weight of $G$, then we denote by $\overline{(Q,\vhi)}$ the $G$-conjugacy class of weights of $G$.
Sometimes we also write  $\overline{(Q,\vhi)}$ simply $(Q,\vhi)$  if no confusion can arise.

If $B$ is a union of blocks of $G$, then we denote by $\Alp^0(B)$ the set of  weights of $G$ belonging to some block in $B$,
while $\Alp(B)$ denotes the set of $G$-conjugacy classes of weights in $\Alp^0(B)$.
In addition, we also denote the set of irreducible Brauer characters of $G$ belonging to some block in $B$ by $\IBr(B)$. 
Sometimes we also write $\IBr(B)$  (resp. $\Alp(B)$) as $\IBr(G,B)$ (resp. $\Alp(G,B)$) to indicate the group $G$.

\subsection{Modules} We will recall some notation on modules over group algebras from \cite{BDR17} and \cite{Lin18}.
Let $(K,\cO,k)$ be an $\ell$-modular system, i.e., $\cO$ is a complete discrete valuation ring having a residue field $k$ of prime characteristic $\ell$ and a quotient field $K$ of characteristic zero.
Assume further that $k$ is algebraically closed and $K$ contains all roots of unity whose order divides the exponent of the finite groups considered.
Let $\Lambda\in\{\cO,k\}$.

For a $\Lambda$-algebra $A$ which is finitely generated and projective as a $\Lambda$-module, we write $A$-mod for the category of left $A$-modules that are  finitely generated and projective as a $\Lambda$-modules.
We denote by $A^{\opp}$ the opposite algebra of $A$.
Following \cite[2.A]{BDR17}, we denote the category of bounded complexes of $A$-mod by $\mathrm{Comp}^b(A)$,  and denote the homotopy category of $\mathrm{Comp}^b(A)$ by $\mathrm{Ho}^b(A)$.
In addition, $D^b(A)$ denotes the bounded derived category of $A$-mod.

For $\mathcal C\in \mathrm{Comp}^b(A)$ there exists a complex $\mathcal C^{\mathrm{red}}$ such that $\mathcal C\cong \mathcal C^{\mathrm{red}}$ in $\mathrm{Ho}^b(A)$ and  $\mathcal C^{\mathrm{red}}$ has no non-zero summand which is homotopy equivalent to $0$.
That is, $\mathcal C\cong \mathcal C^{\mathrm{red}}\oplus \mathcal C_0$ where $\mathcal C_0$ is homotopy equivalent to 0.
Denote by $\End_A^\bullet(\mathcal C)$ the total Hom-complex, with degree $n$ term $\bigoplus_{j-i=n}\mathrm{Hom}(C^i,C^j)$.

Let $B$ be a $\Lambda$-algebra, finitely generated and projective as a $\Lambda$-module.
Let $\mathcal C$ be a bounded complex of $A$-$B$-bimodules, finitely generated and projective as left $A$-modules and as right $B$-modules.
We say that the complex $\mathcal C$ induces a \emph{Rickard equivalence} between the algebras $A$ and $B$ if the canonical map $A\to \End^\bullet_{B^{\opp}}(\mathcal C)^{\opp}$ is an isomorphism in the homotopy category $\mathrm{Ho}^b(B\otimes_\Lambda B^{\opp})$ and the canonical map  $B\to \End^\bullet_A(\mathcal C)$ is an isomorphism in the homotopy category $\mathrm{Ho}^b(A\otimes_\Lambda A^{\opp})$.
Note that by \cite[\S2.1]{Ri96}, $A$ and $B$ are Rickard equivalent if and only if they are derived equivalent.
Let $M$ be an $A$-$B$-bimodule.
Then $M$ induces a Morita equivalence between algebras $A$ and $B$ if and only if the complex $M[0]$ induces a Rickard equivalence between  $A$ and $B$.

\subsection{Local representations}

Let $G$ be a finite group.
We denote by $G^{\opp}$ the opposite group to $G$ and set $$\Delta G=\{ (g,g^{-1})\mid g\in G \}\subseteq G\ti G^{\opp}.$$

We give an elementary lemma, which will be needed in the sequel.

\begin{lem}\label{lem:ext-abel}
	Let $N$ be a normal subgroup of $G$ with $E=G/N$ an  $\ell'$-group.
	Let $M$ be a $G$-stable $\Lambda N$-module such that there exists a $\Lambda G$-module $M'$ extending $M$.
	Then \begin{enumerate}[\rm(i)]
		\item $M'$ is a direct summand of $\Ind_N^G M$, 
		\item $\Ind_N^G M\simeq \Res^{G\ti G^{\opp}}_{\Delta G}(M'\otimes\Lambda E)$, and
		\item if $\Lambda G$-module $M''$ is an extension of $M$, then $M''\simeq M'\otimes V$ where $V$ is a  $\Lambda E$-module with $\dim_{\Lambda} V=1$.
	\end{enumerate}	
\end{lem}

\begin{proof}
	By Mackey's formula (see, \emph{e.g.}, \cite[Thm.~2.4.1]{Lin18}),
	$$\Ind_N^G \Res_N^G M'\simeq \Res^{G\ti G^{\opp}}_{\Delta G} \Ind^{G\ti G^{\opp}}_{G\ti N^{\opp}}(M'\otimes\Lambda).$$
	Here $\Lambda$ denotes the trivial $\Lambda N^{\opp}$-module.
	Thus we have (ii).
	Since $E$ is an $\ell'$-group, by Maschke’s theorem, $\Lambda E$ is semisimple and then
	$$\Lambda E\simeq V_1\oplus\cdots\oplus V_t,$$ where $V_i$ are simple $\Lambda E$-modules such that $V_1$ is the trivial module.
	So $$\Ind_N^G M\simeq \bigoplus_{i=1}^t M'\otimes V_i.$$
	In particular, $M'$ is a direct summand of $\Ind_N^G M$ since $M'\simeq M'\otimes V_1$ as $\Lambda G$-modules.
	If $\Lambda G$-module $M''$ is an extension of $M$, then $M''\simeq M'\otimes V$ where $V$ is a  $\Lambda E$-module with $\dim_{\Lambda} V=1$.
\end{proof}

Let $M$ be a $\Lambda G$-module and $H$ be a subgroup of $G$.
We denote by $M^H$ the set of $H$-fixed points in $M$.
For subgroups $H\le L \le G$, recall that the relative trace map $\Tr_H^L:M^H\to M^L$ is defined by $\Tr_H^L(m)=\sum\limits_{h\in[L/H]}h\cdot m$ for all $m\in M^H$.
Here $[L/H]$ denotes a complete set of representatives of the left $H$-cosets in $L$.

For an $\ell$-subgroup $Q$ of $G$, we consider the  Brauer functor
$\Br_Q^G: \Lambda G\text{-mod}\to k\N_G(Q)\text{-mod}$, where for a $\Lambda G$-module $M$, 
$$\Br_Q^G(M)=k\otimes_\Lambda (M^Q/\sum_{P<Q}\Tr_P^Q(M^P)).$$
See for example \cite[\S 5.4]{Lin18}.
Since $\Br^G_Q$ is an additive functor,  it respects homotopy equivalences. In this way it induces a functor
$$\Br_Q^G:\mathrm{Ho}^b(\Lambda G)\to\mathrm{Ho}^b(k \N_G(Q)/Q).$$
We also denote $\Br_Q^G$ briefly by $\Br_Q$ since $$\Br^H_Q\circ \Res^G_H=\Res_{\N_H(Q)}^{\N_G(Q)}\circ \Br^G_Q$$ for any subgroup $H$ of $G$ containing $Q$. 

Denote by $\mathrm{br}_Q:(\Lambda G)^Q\to k\C_G(Q)$ the algebra morphism given by $$\mathrm{br}_Q(\sum_{g\in G}\lambda_g g)=\sum_{g\in \C_G(Q)}\lambda_g g,$$
where $\lambda_g\in \Lambda$.

An \emph{$\ell$-permutation $\Lambda G$-module} is defined to be a direct summand of a finitely generated permutation $\Lambda G$-module.
We denote the full subcategory of $\Lambda$-mod with objects the $\ell$-permutation $\Lambda G$-modules
by $\Lambda G$-perm.
Let $H\le G$ and $\mathcal C\in\mathrm{Comp}^b(\Lambda G\ti(\Lambda H)^{\opp})$.
Following \cite{Ri96}, we say that $\mathcal C$ is \emph{splendid} if  $\mathcal C^{\mathrm{red}}$ is a complex of $\ell$-permutation $\Lambda G\ti(\Lambda H)^{\opp}$-modules whose indecomposable direct summands have a vertex contained in $\Delta H$.
Let $b\in \Z(\Lambda G)$ and $c\in \Z(\Lambda H)$ be  idempotents. 
If $b\mathcal C c$ is splendid and induces a Rickard equivalence between $\Lambda G b$ and $\Lambda Hc$, then we say that $b\mathcal C c$ induces a \emph{splendid Rickard equivalence} between $\Lambda G b$ and $\Lambda Hc$.
See \cite[\S 9]{Lin18} for more about equivalence between blocks of group algebras.

\subsection{Brauer categories}

We will make use of the notion of Brauer pairs.
A \emph{Brauer pair} of $G$ is a pair $(Q,b_Q)$, where $Q$ is an $\ell$-subgroup of $G$ and $b_Q$ is an $\ell$-block of $\C_G(Q)$.
There is an order relation ``$\le$" among the Brauer pairs of $G$; see for example \cite[\S 6]{Lin18}.
Let $B$ be an $\ell$-block of $G$, then a Brauer pair $(Q,b_Q)$ is called a \emph{$B$-Brauer pair} if $(1,B)\le (Q,b_Q)$.
An $\ell$-subgroup $D$ of $G$ is a defect group of $B$ if and only if there exists a $B$-Brauer pair $(D,b_D)$ which is maximal with respect to ``$\le$".
In addition, all maximal $B$-Brauer pairs are $G$-conjugate.

For a block $B$ of finite group $G$, we denote by $\bBr(G,B)$ the Brauer category (see for instance \cite[\S 47]{Th95}), whose objects are the $B$-Brauer pairs and the set of morphisms from $(Q,b_Q)$ to $(P,b_P)$ containing all homomorphism from $Q$
 to $P$ induced by conjugation via some element $g$ of $G$ such that $(Q,b_Q)^g\le (P,b_P)$.
If $(D,b_D)$ is a maximal $B$-Brauer pair, we denote by $\cF_{(D,b_D)}(G,B)$ the full subcategory of $\bBr(G,B)$ whose objects are the $B$-Brauer pairs contained in $(D,b_D)$.
By \cite[Thm.~47.1]{Th95}, the inclusion $\cF_{(D,b_D)}(G,B)\to \bBr(G,B)$ induces an equivalence between these two categories.

Let $H$ be a subgroup of $G$.
Let $b\in \Z(\Lambda G)$ and $c\in\Z(\Lambda H)$ be primitive idempotents. 
Suppose that there exists a complex $\mathcal C\in\mathrm{Comp}^b(\Lambda G b\ti (\Lambda Hc)^{\opp})$ inducing a splendid Rickard equivalence between $\Lambda Gb$ and $\Lambda Hc$.
Then by a theorem of Puig \cite[Thm.~19.7]{Pu99}, the Brauer categories $\bBr(G,b)$ and $\bBr(H,c)$ are equivalent.
In particular, the blocks $b$ and $c$ have a common defect group.

Let  $(Q,b_Q)$ be a Brauer pair of $G$.
Denote by $\N_G(Q,b_Q)$ the stabilizer of $b_Q$ in $\N_G(Q)$.
By \cite[Exercise 40.2(b)]{Th95}, $b_Q$ is also a block idempotent of $\N_G(Q,b_Q)$. 
So $\Tr_{\N_G(Q,b_Q)}^{\N_G(Q)} (b_Q)$ is the unique block of $\N_G(Q)$ covering $b_Q$.

We denote by $\bNBr(G,B)$ the set of pairs $(Q,B_Q)$, where $Q$ is an $\ell$-subgroup of $G$ and $B_Q$ is a block of $\N_G(Q)$ with $(B_Q)^G=B$.

\subsection{Varieties}

Let $\bX$ be a quasi-projective variety over $\overline{\F}_p$ and a finite group $G$ acts on $\bX$.
Then there is an object $\G\Ga_c(\bX,\Lambda)$ in $\mathrm{Ho}^b(\Lambda G\textrm{-perm})$, which is unique up to isomorphism; see  \cite[2.C]{BDR17}. 
We denote the image of  $\G\Ga_c(\bX,\Lambda)$ in $D^b(\Lambda G)$ by $\R\Ga_c(\bX,\Lambda)$  (the global section complex).
Moreover,  if $d$ is a positive integer, then
we denote by $H^d_c(\bX,\Lambda)\in \Lambda G$-mod the $d$-th cohomology module of the complex $\R\Ga_c(\bX,\Lambda)$.
For a variety $\bX$ is of dimension $n$,
we write 
$H^{\dim}_c(\bX,\Lambda)=H^n_c(\bX,\Lambda)$.

Let $Q$ be an $\ell$-subgroup of $G$.
Then by \cite[Thm.~4.2]{Ri94},  the inclusion $\bX^Q\hookrightarrow \bX$ induces an isomorphism
$$\G\Ga_c(\bX^Q,k)\to \Br_Q(\G\Ga_c(\bX,\Lambda))$$
in $\mathrm{Ho}^b(k\N_G(Q)\textrm{-perm})$.

\section{Modular character triples and equivariant bijections}
\label{sec:char-triple}

In \cite{Sp17} the inductive conditions for some of the local-global conjectures
were rephrased in terms of (modular) character triples.
In order to consider the inductive BAW condition, we start from recalling the isomorphisms of (modular) character triples. 
Let $G$ be a finite group and
$N\unlhd G$.
We say $(G,N,\theta)$ a \emph{character triple} (resp. \emph{modular character triple}) if $\theta\in\Irr(N)$ (resp. $\theta\in\IBr(N)$) and $\theta$ is $G$-invariant.
We denote the set of characters (resp. Brauer characters) $\chi$ of $G$ such that $\Res_{N}^G\chi$ is a multiple of $\theta$ by $\Char(G\mid\theta)$ (resp. $\BrCh(G\mid\theta)$).
This is the set of nonnegative integer linear combinations of the elements of $\Irr(G\mid\theta)$ (resp. $\IBr(G\mid\theta)$).

\subsection{Isomorphism between modular character triples}

Isomorphisms between two modular character triples can be found in 
\cite[\S8]{Na98}.
Let $(G,N,\theta)$ and $(H,M,\varphi)$ be two modular character triples and $\tau:G/N\to H/M$ be an isomorphism. For $N\le L\le G$, let $L^\tau$ denote the inverse image in $H$ of $\tau(L/N)$. For every such $L$, suppose there exists a map $\sigma_L:\BrCh(L\mid \theta)\to\BrCh(L^\tau\mid \varphi)$ such that for $N\le K\le L \le G$ and $\chi,\psi\in \BrCh(L\mid \theta)$, the following holds:
	\begin{enumerate}[\rm(a)]
\item  $(\sigma_L)|_{\IBr(L\mid \theta)}:\IBr(L\mid \theta)\to\IBr(L^\tau\mid \varphi)$ is bijective,
\item  $\sigma_L(\chi+\psi)=\sigma_L(\chi)+\sigma_L(\psi)$,
\item  $\sigma_K(\Res^L_K\chi)=\Res^{L^\tau}_{K^\tau}\sigma_L(\chi)$,
\item $\sigma_L(\chi\beta)=\sigma_L(\chi)\beta^\tau$ for $\beta\in\IBr(L/N)$.
	\end{enumerate}
	Let $\sigma$ denote the union of the maps $\sigma_L$. Then $(\sigma,\tau)$ is called an \emph{isomorphism} between modular character triples $(G,N,\theta)$ and $(H,M,\varphi)$.

\begin{rmk}
\begin{enumerate}[(i)]	
	\item To construct an isomorphism $(\sigma,\tau)$ it suffices to define $\sigma_L$ between $\IBr(L\mid \theta)$ to $\IBr(L^\tau\mid \varphi)$ to be bijective, then extend the definition by (b) and check (c) and (d) hold for $\chi\in \IBr(L\mid \theta)$.
	\item If $(\sigma,\tau):(G,N,\theta)\to(H,M,\varphi)$ is an isomorphism between modular character triples, then for every $N\le J\le H$, $\chi\in\IBr(J\mid\theta)$, one has $$\frac{\chi(1)}{\theta(1)}=\frac{\sigma_J(\chi)(1)}{\vhi(1)}.$$
	In particular, $\theta$ extends to $J$ if and only if $\vhi$ extends to $J^\tau$.
\end{enumerate}	
\end{rmk}

Let $(G,N,\theta)$ be a modular character triple. If For $N\le L\le G$, $\psi\in\BrCh(L\mid \theta)$ and $\bar{g}=gN\in G/N$ we define ${}^{\bar{g}}\psi\in\BrCh({}^gL\mid \theta)$ by ${}^{\bar{g}}\psi(^gh)=\psi(h)$ for $h\in L$.

Following \cite[p.~281]{SV16}, we say an isomorphism $(\sigma,\tau):(G,N,\theta)\to (H,M,\varphi)$ is \emph{strong} if $^{\tau(\bar{g})}(\sigma_L(\psi))=\sigma_{^gL}(^{\bar{g}}\psi)$ for all $\bar{g}\in G/N$, all $L$ with  $N\le L\le G$ and all $\psi\in\BrCh(L\mid \theta)$.

\begin{rem}
 Isomorphism and strong isomorphism are equivalence relations on modular character triples.
\end{rem}

The isomorphism and strong isomorphism can be also defined for ordinary case in the same way; see \cite[\S 11]{Is76}.
They share similar properties with the modular case.

We will make use of the notion of ordinary-modular character triples; see \cite[p.~190]{Na98}.
Let $N\unlhd G$, and let $\theta\in\Irr(N)$ be $G$-invariant.
If $\theta^0\in\IBr(N)$, we say $(G,N,\theta)$ is an \emph{ordinary-modular character triple}.
We say two ordinary-modular character triples 
$(G,N,\theta)$ and $(H,M,\varphi)$ are \emph{isomorphic} if there is an ordinary character triple isomorphism $$(\mu,\tau):(G,N,\theta)\to(H,M,\varphi)$$
and a modular character triple isomorphism $$(\sigma,\tau):(G,N,\theta^0)\to(H,M,\varphi^0)$$
such that for every $N\le L\le G$ and $\chi\in\Char(L\mid\theta)$ one has $\mu_L(\chi)^0=\sigma_L(\chi^0)$.
We will say in this case that $(\mu,\sigma,\tau)$ is an isomorphism between the ordinary-modular character triples $(G,N,\theta)$ and $(H,M,\varphi)$.
In addition, $(\mu,\sigma,\tau)$ is called to be \emph{strong} if both $(\mu,\tau)$ and $(\sigma,\tau)$ are strong.

In the following, when considering the isomorphism  between (modular) character triple $(G,N,\theta)$ and $(H,M,\varphi)$, we always assume that $H\le G$, $G=NH$ and $M=H\cap N$.
Moreover, $\tau$ will be chosen to be the inverse of the natural isomorphism $H/M\to G/N$ induced by the containing $H\le G$.
In this way, we may abbreviate $(\sigma,\tau)$ to $\sigma$.

\subsection{Block isomorphisms}

Let $(G,N,\theta)$ be a modular character triple.
For the background of projective representations, we refer to \cite[Chap. 3, Sect. 5]{NT11}.
A projective representation $\cP:G\to\GL_{\theta(1)}(k)$ is said to be \emph{associated to} $(G,N,\theta)$ if the restriction $\cP|_N$ affords $\theta$ and $\cP(ng)=\cP(n)\cP(g)$ and $\cP(gn)=\cP(g)\cP(n)$ for all $n\in N$ and $g\in G$.

We recall our favorite isomorphisms between modular character triples from \cite{Sp17,Sp18}.
Let $(G,N,\theta)$ and $(H,M,\vhi)$ be two modular character triples with $H\le G$.
We write $(G,N,\theta)\geqslant(H,M,\vhi)$ if 
\begin{itemize}
	\item $G=NH$ and $M=N\cap H$,
	\item there are projective representations $\cP$ and $\cP'$ associated to $(G,N,\theta)$ and $(H,M,\vhi)$  with factor sets $\alpha$ and $\alpha'$, respectively, such that $\alpha|_{H\ti H}=\alpha'$.
\end{itemize}
We say that $(G,N,\theta)\geqslant(H,M,\vhi)$ is given by $(\cP,\cP')$.

By \cite[Thm.~3.1]{SV16}, one has

\begin{thm}\label{thm:triple-isomor}
	Let $(G,N,\theta)\geqslant(H,M,\vhi)$ be given by $(\cP,\cP')$.
	Then for every intermediate subgroup $N\le J\le G$,  the  map $\sigma_J:\BrCh(J\mid \theta)\to\BrCh(J\cap H\mid\vhi)$ given by
	$$\mathrm{trace}(\mathcal Q\otimes \cP|_J)\mapsto \mathrm{trace}(\mathcal Q\otimes \cP'|_{J\cap H}),$$
	for any projective representation $\mathcal Q$ of $J/N\cong (J\cap H)/M$ whose factor set is inverse to the one of $\cP|_{J}$,
	is a well-defined bijection and
	induces a strong isomorphism between the modular character triples $(G,N,\theta)$ and $(H,M,\vhi)$. 
\end{thm}
We say that the $\sigma$ in Theorem~\ref{thm:triple-isomor} is \emph{induced} by $(\cP,\cP')$.

Let $(G,N,\theta)\geqslant(H,M,\vhi)$ be given by $(\cP,\cP')$.
Then following Definition 3.3 of \cite{SV16},
we write  $(G,N,\theta)\geqslant_c(H,M,\vhi)$ if $\C_G(N)\le H$ and for every $c\in\C_G(N)$  the scalars associated to $\cP(c)$ and $\cP'(c)$ coincide.

Let $(G,N,\theta)\geqslant_c(H,M,\vhi)$ be given by $(\cP,\cP')$.
Following Definition 4.19 of \cite{Sp18},
we write  $(G,N,\theta)\geqslant_b(H,M,\vhi)$ if further
\begin{itemize}
	\item there is a defect group $D$ of $\bl(\vhi)$ such that $\C_G(D)\le H$, and
	\item the maps $\sigma_J$ induced by $(\cP,\cP')$ satisfy $\bl_J(\psi)=\bl_{J\cap H}(\sigma_J(\psi))^J$ for every $N\le J\le G$ and $\psi\in\IBr(J\mid\theta)$.
\end{itemize}
Then we say that $(G,N,\theta)\geqslant_b(H,M,\vhi)$ is given by $(\cP,\cP')$.
``$\geqslant_b$" is  called the \emph{block isomorphism} between modular character triples. 

\begin{lem}\label{lem:similar-proj}
	Suppose that  $(\cP,\cP')$  gives $(G,N,\theta)\geqslant_b (H,M,\vhi)$.
	Let $A\in\GL_{\theta(1)}(k)$, $A'\in\GL_{\vhi(1)}(k)$ and let $\overline{\mu}:G/N\to k^\ti$ be a map.
	Then $(\mu A\cP A^{-1},\mu_H A'\cP'{A'}^{-1})$ also gives $(G,N,\theta)\geqslant_b (H,M,\vhi)$,  where $\mu$ is the lift of $\overline \mu$ to $G$.
	
\end{lem}

\begin{proof}
	This can be checked by the definition directly.
\end{proof}

\begin{lem}\label{lem:equ-transitive}
	Let $(G,N,\theta)\geqslant_b (H,M,\vhi)$ and 
	$(H,M,\vhi)\geqslant_b (H_1,M_1,\vhi_1).$
Assume that	$\bl(\vhi_1)$ has a defect group $D_1$ such that $\C_G(D_1)\le H_1$.
	Then $(G,N,\theta)\geqslant_b(H_1,M_1,\vhi_1).$	
\end{lem}

\begin{proof}
	Suppose that $(G,N,\theta)\geqslant_b (H,M,\vhi)$ and 
	$(H,M,\vhi)\geqslant_b(H_1,M_1,\vhi_1)$ are given by
	the projective representations $(\cP,\cP')$ and $(\cP_1,\cP'_1)$  respectively.
	We may assume that $\cP'=\cP_1$.
	In fact,  by \cite[Chap.~3, Thm.~5.7]{NT89},
	there exist a matrix $A\in\GL_{\vhi(1)}(k)$ and a map $\overline \mu:H/M\to k^\ti$ such that $\cP'=\mu A \cP_1 A^{-1}$, where $\mu$ is the lift of $\overline \mu$ to $H$.	
	By Lemma~\ref{lem:similar-proj}, $(\mu A\cP_1A^{-1},\mu_{H_1} \cP'_1)$ also gives $(H,M,\vhi)\geqslant_b(H_1,M_1,\vhi_1)$.
	Then it can be checked directly that $(G,N,\theta)\geqslant_b(H_1,M_1,\vhi_1)$ by the definition. 
\end{proof}

\begin{rmk}
The relations $\geqslant$, $\geqslant_c$, $\geqslant_b$ for ordinary case can be defined in the same way (see  \cite{Sp18}) and the properties given above also holds.
\end{rmk}

	Let $G$ be a finite group and $B$ a block of $G$.
Assume that for $\Ga:=\Aut(G)_B$ we have
\begin{itemize}
	\item there is a $\Ga$-equivariant bijection $\Omega:\IBr(B)\to \Alp(B)$, and 
	\item for any $\psi\in\IBr(B)$ and $\overline{(Q,\vhi)}=\Omega(\psi)$,
	\begin{equation*}\label{equ:ibaw-bij}
	(G\rtimes \Ga_\psi,G,\psi)\geqslant_b((G\rtimes\Ga)_{Q,\vhi},\N_G(Q),\vhi^0).
		\addtocounter{thm}{1}\tag{\thethm}
	\end{equation*}
\end{itemize}
Then we say that $\Omega:\IBr(B)\to \Alp(B)$ is an \emph{iBAW-bijection} for $B$.

In \cite{Sp17} the inductive conditions for Alperin weight conjecture
was rephrased in terms of character triples and we recall as follows.
Let $S$ be a finite non-abelian simple group, and let $\ell$ be a prime dividing $|S|$.
Suppose that $G$ is the universal $\ell'$-covering group of $S$, and $B$ is an $\ell$-block of $G$.
By  \cite[Thm.~4.4]{Sp17}, $B$ is said to satisfy the \emph{inductive blockwise Alperin weight (BAW) condition}  if for $\Ga=\Aut(G)_B$,
\begin{itemize}
	\item there exists a $\Ga$-equivariant bijection $\Omega_B:\IBr(B)\to\Alp(B)$, and 
	\item for every $\psi\in\IBr(B)$ and $\overline{(Q,\vhi)}=\Omega_B(\psi)$, one has 
	\begin{equation*}	\label{iso:quotient}
	(G/Z\rtimes \Ga_\psi,G/Z,\overline \psi)\geqslant_b( (G/Z\rtimes \Ga)_{Q,\vhi}, \N_G(Q)/Z,\overline{\vhi^0} ),
	\addtocounter{thm}{1}\tag{\thethm}
	\end{equation*}
	where $Z=\ker(\psi)\cap Z(G)$, and $\overline\psi$ and $\overline{\vhi^0}$ lift to $\psi$ and $\vhi^0$ respectively.
\end{itemize}
We also say that $B$ is  \emph{BAW-good} if $B$ satisfies the inductive BAW condition.
We say the simple group \emph{$S$ is BAW-good for the prime $\ell$} if every $\ell$-block of $G$ is BAW-good.
We say the simple group  \emph{$S$ satisfies the inductive BAW condition} (or \emph{$S$ is BAW-good} ) if it is BAW-good for any prime dividing $|S|$. \label{def-iBAW}

We note that by a similar result as \cite[Lemma~3.2]{NS14} for Brauer characters,
(\ref{equ:ibaw-bij}) implies (\ref{iso:quotient}). Thus
if $B$ is BAW-good, then there is an iBAW-bijection for $B$.

\subsection{An isomorphism  between ordinary-modular character triples}

We first  recall the Dade--Glauberman--Nagao correspondence from \cite{NS14}.
Let $N\unlhd M$ be finite groups such that $M/N$ is an $\ell$-group and $D_0\subseteq \Z(M)$ is  a normal $\ell$-subgroup of $N$.
Suppose that $b$ is an $M$-invariant block of $N$ which has  defect group $D_0$.
Let $B$ be the (unique) block of $M$ covering $b$ and
let $B'$ be the Brauer correspondent of $B$.
Suppose that $D$ is a defect group of $B$.
Denote $L:=\N_N(D)$ and let $b'$ be the unique block of $L$ covered by $B'$.
According to \cite[Thm.~5.2]{NS14}, there is a natural bijection $\pi_D:\Irr_D(b)\to \Irr_D(b')$, which is
called the \emph{Dade--Glauberman--Nagao (DGN) correspondence}.
Here, $\Irr_D(b)$ denotes the set of $D$-invariant irreducible characters in $b$.

\begin{lem}\label{lem:Lad10}
	Suppose that $N\unlhd G$ and, in addition, that $M/N$ is a normal $\ell$-subgroup of $G/N$.
	Let $\theta\in\dz(N)$ with $G_\theta=G$. 
	Let $D$ be a defect group of the unique block of $M$ covering $\bl(\theta)$ and $\theta_0:=\pi_D(\theta)$ be the DGN-correspondent of $\theta$.
	Then there exists an $\N_G(D)$-equivariant bijection
	$$\Pi:\Irr(M\mid\theta)\to\Irr(\N_M(D)\mid \theta_0)$$ such that 
	$(G_{\theta'},M,\theta')\geqslant_b(\N_G(D)_{\theta'},\N_M(D),\Pi(\theta'))$
	for every $\theta'\in\Irr(M\mid\theta)$.
\end{lem}

\begin{proof}
	$G=N\cdot\N_G(D)$ is obvious and
	this assertion follows by \cite[Thm.~5.13]{NS14}.
\end{proof}

We give the following result on extension of representations.

\begin{lem}\label{lem:extension-ring}
Let $G$ be a finite group and $N\unlhd G$.
Let $\theta\in\dz(N)$ such that $G_\theta=G$ and $\theta$ extends to $G$.
Let $\cD:N\to\GL_{\theta(1)}(\cO)$ be an $\cO$-representation of $N$ affording $\theta$.
Then 
\begin{enumerate}[\rm(i)]
\item there is an $\cO$-representation $\tcD$ of $G$ with $\tcD|_N=\cD$, and
\item if $\tcD':\G\to\GL_{\theta(1)}(K)$ is a $K$-representation of $G$ with $\tcD'|_{N}=\cD$, then $\tcD'(G)\subseteq \GL_{\theta(1)}(\cO)$.
\end{enumerate}
\end{lem}

\begin{proof}
(i) Let $\ttheta\in\Irr(G)$ be an extension of $\theta$ and let
 $\hcD$ be an $\cO$-representation of $G$ affording $\ttheta$.
Then $\hcD|_N$ afford $\theta$.
Since $\theta$ is of defect zero,
there exists $O\in\GL_{\theta(1)}(\cO)$ such that $O(\hcD|_N) O^{-1}=\cD$.
Let $\tcD=O\hcD O^{-1}$.
Then $\tcD$ is an $\cO$-representation of $G$ with
$\tcD|_N=\cD$.

(ii)  Let $\tcD':\G\to\GL_{\theta(1)}(K)$ be a $K$-representation of $G$ with $\tcD'|_{N}=\cD$.
Then there exists a linear character $\lambda\in\Irr(G/N)$ such that $\tcD'=\lambda\tcD$.
So $\tcD'$ is also an $\cO$-representation of $G$ and this proves (ii).
\end{proof}

We will make use of the following lemma.

\begin{lem}\label{lem-proj-cO}
Let $(G,N,\theta)$ be a character triple and let $\cP$ be a projective representation of $G$ over $K$ which is associated to $(G,N,\theta)$ and has factor set $\alpha$.
If $\alpha(g,g')\in\cO$ for every $g,g'\in G$, then there exists $A\in\GL_{\theta(1)}(K)$ such that $A\cP(g)A^{-1}\in\GL_{\theta(1)}(\cO)$ for every $g\in G$.
\end{lem} 

\begin{proof}
Denote by the  $K^\alpha G$ the generalized group ring of $G$ with factor set $\alpha$ defined as in \cite[p.~214]{NT89}.	
	Then $\mathcal P$ induces a linear $K$-representation of the $K$-algebra $K^\alpha G$.
	This gives a linear $K$-representation of the $\cO$-algebra $\cO^\alpha G$,
	and by  \cite[Chap. 2, Thm.~1.6]{NT89},
	it is $K$-equivalent to a  linear $\cO$-representation of $\cO^\alpha G$. 
	From this, there exists $A\in\GL_{\theta(1)}(K)$ such that $A\cP(g)A^{-1}\in\GL_{\theta(1)}(\cO)$ for every $g\in G$
\end{proof}

Now we show the following block isomorphism between modular character triples.

\begin{prop}\label{lem:Ladisch}
	Suppose that $N\unlhd G$ and, in addition, that $M/N$ is a normal $\ell$-subgroup of $G/N$.
	Let $\theta\in\dz(N)$ with $G_\theta=G$. 
	Assume that $G=XL$, where $M\subseteq X\unlhd G$ and $L\le G$ satisfy that $X\cap L=N$ and
	the quotient group $X/N$ is abelian, and $\theta$ extends to both  $X$ and $L$.	
	Let $D$ be a defect group of the unique block of $M$ covering $\bl(\theta)$ and $\theta_0:=\pi_D(\theta)$ be the DGN-correspondent of $\theta$.
	Then 
 $$(G, M, \widehat \theta)\geqslant_b(\N_G(D), \N_M(D), \widehat \theta')$$ for some $\widehat \theta\in\IBr(M\mid\theta^0)$ and $\widehat \theta'\in \IBr(\N_M(D)\mid(\theta_0)^0)$.	

\end{prop}

\begin{proof}
First by Lemma~\ref{lem:Lad10}, there exists an $\N_G(D)$-equivariant bijection
$$\Pi:\Irr(M\mid\theta)\to\Irr(\N_M(D)\mid \theta_0)$$ such that 
$(G_{\theta'},M,\theta')\geqslant_b(\N_G(D)_{\theta'},\N_M(D),\Pi(\theta'))$
for every $\theta'\in\Irr(M\mid\theta)$.
Let $B$ be the block of $M$ covering $\bl(\theta)$ and $B'$ the block of $\N_M(D)$ covering $\bl(\theta_0)$.
Since $D$ is the defect group of $B'$ and
$\N_M(D)=\C_N(D)D=\C_M(D)D$, we know that $B'$ is of self-centralizing.
We fix $\theta'\in\Irr(M\mid\theta)$ and $\theta_0'=\Pi(\theta')$ such that $\theta_0'$ is the canonical 	character of $B'$.
Then $\theta'_0$ is $\N_G(D)$-invariant, while $\theta'$ is $G$-invariant.
In addition, $\theta_0'$ is an extension of $\theta_0$, while $\theta'$ is an extension of $\theta$.
Now
$\theta$ extends to both  $X$ and $L$,
by \cite[Cor.~4.2]{Is84}, $\theta'$ extends to both $X$ and $ML$.
Thus $\theta'_0$ extends to both $\N_X(D)$ and $\N_{ML}(D)$.
Let $(\cP_0,\cP_0')$ be projective representations (over $\mathbb C$)
giving  $(G,M,\theta')\geqslant_b(\N_G(D),\N_M(D),\theta'_0)$ and
let $\sigma$ be the isomorphism between
the character triples 
$(G, M, \theta')$ and $(\N_G(D), \N_M(D), \theta_0')$ given by $(\cP_0,\cP_0')$.

Let $\psi$ be an extension of $\theta'$ to $X$, and $\chi$ be an extension of $\theta'$ to $ML$.
	Let $\cD$ be an $\cO$-representation of $M$ affording $\theta'$.
	By Lemma~\ref{lem:extension-ring},
there exist an $\cO$-representation $\hcD$ of $X$ affording $\psi$ and $\tcD$ an $\cO$-representation of $L$ affording $\chi$ such that $\hcD|_M=\tcD|_M=\cD$.
By the proof of \cite[Lemma 2.11]{Sp12}, we define a projective representation $\cP:G\to\GL_{\theta'(1)}(K)$ of $G$
satisfies that 
$$\cP(xl)=\hcD(x)\tcD(l)\ \text{for}\ x\in X \ \text{and}\ l\in ML.$$
Then $\cP|_{M}=\cD$.
According to \cite[Chap. 3, Thm.~5.7]{NT89},  there  $\cP=\mu A\cP_0A^{-1}$, where  $A\in\GL_{\theta'(1)}(\mathbb C)$
and $\mu$ is the lift of a map $\overline\mu:G/M\to\mathbb C^\ti$.

Note that $\theta'_0$ extends to both $\N_X(D)$ and $\N_{ML}(D)$.
Let $\psi_0$ be an extension of $\theta'_0$ to $\N_X(D)$ such that $\psi_0=\sigma_X(\psi)$ and $\chi_0$ be an extension of $\theta'_0$ to $\N_{ML}(D)$ such that $\chi_0=\sigma_{ML}(\chi)$.
Let $\cP''=\mu|_{\N_G(D)}\cP'_0$.
Then by the definition of $\sigma$, $\cP''|_{\N_X(D)}$ is a $\mathbb C$-representation of $\N_X(D)$ affording $\psi_0$ and $\cP''|_{\N_{ML}(D)}$ is a $\mathbb C$-representation of $\N_{ML}(D)$ affording $\chi_0$.
In addition, $\cP''|_{\N_M(D)}$ affords $\theta'_0$.
 Let $\cD'$ be an $\cO$-representation of $\N_M(D)$ affording $\theta_0'$.
Then there exists $A'\in\GL_{\theta_0'(1)}(\mathbb C)$ such that $\cD'=A'\cP''|_{\N_M(D)}{A'}^{-1}$.
Thus $A'\cP''|_{\N_X(D)}{A'}^{-1}$ and $A'\cP''|_{\N_{ML}(D)}{A'}^{-1}$ are extensions of $\cD'$ and by Lemma~\ref{lem:extension-ring}
we know that they are $\cO$-representations.
Let $\hcD'=A'\cP''|_{\N_X(D)}{A'}^{-1}$ and 
$\tcD'=A'\cP''|_{\N_{ML}(D)}{A'}^{-1}$.
Then we can define a projective representation 
$\cP':\N_G(D)\to\GL_{\theta_0'(1)}(\cO)$ by
$$\cP'(xl)=\hcD'(x)\tcD'(l)\ \text{for}\ x\in \N_X(D)\ \text{and}\ l\in \N_{ML}(D).$$
From this, it can be checked directly that $(\cP,\cP')$ also gives $$(G,M,\theta')\geqslant_b(\N_G(D),\N_M(D),\theta'_0).$$
by a similar argument as in \cite[Lemma 2.13]{Sp12}.

Denote by $^*:\cO\to k$ the canonical epimorphism.
Now $\widehat \theta={\theta'}^0$ and $\widehat \theta'=(\theta'_0)^0$.
In this way $\cD^*$, $\tcD^*$, $\hcD^*$ affords $\widehat \theta$, $\chi^0$, $\psi^0$ respectively.
Similar as  \cite[Lemma 2.11]{Sp12}, $\cP^*:G\to\GL_{\htheta(1)}(k)$ defined by
$$\cP^*(xl)=\hcD(x)^*\tcD(l)^* \ \text{for}\ x\in X\ \text{and}\ l\in ML,$$ is a projective representation of $G$ over $k$ associated to $(G,M,\widehat \theta)$.
In addition, the factor set $\alpha^*$ of $\cP^*$ satisfies that
$$\alpha^*(xlM,x'l'M)=(\la_l)^0(x')\ \text{for $x,x'\in X$ and $l,l'\in ML$}.$$
Here for every $l\in L$, the linear character $\la_l\in\Irr(X/N)$ is determined by
$\psi=\la_l \psi^l$.
Analogously, ${\cP'}^*:\N_G(D)\to\GL_{\htheta'(1)}(k)$ is defined to satisfy $${\cP'}^*(xl)=\hcD'(x)^*\tcD'(l)^*\ \text{for}\ x\in \N_X(D)\ \text{and}\ l\in \N_{ML}(D),$$ is a projective representation of $\N_G(D)$ over $k$ associated to 
$(\N_G(D), \N_M(D), \widehat \theta')$.
The factor set ${\alpha'}^*$ of ${\cP'}^*$ satisfies that
$${\alpha'}^*(xl\N_M(D),x'l'\N_M(D))=({\la'}_l)^0(x')\ \text{for $x,x'\in \N_X(D)$ and $l,l'\in \N_{ML}(D)$}.$$
	Here for every $l\in L$, the linear character $\la_l'\in\Irr(\N_X(D)/\N_N(D))$ is determined by
$\psi_0=\la_l' \psi_0^l$.
As in the proof of \cite[Lemma 3.4]{BS20}, $(\cP^*,{\cP'}^*)$ gives $(G, M, \widehat \theta) \geqslant(\N_G(D), \N_M(D), \widehat \theta')$.
In addition, it can be checked that $(\cP^*,{\cP'}^*)$ gives $(G, M, \htheta) \geqslant_c(\N_G(D), \N_M(D), \widehat \theta')$ immediately.

For the block induction property, we claim that 	the ordinary-modular character triples
$(G,M,\theta')$ and $(\N_G(D),\N_M(D),\theta'_0)$ are	 strong isomorphic via 
$$(G, M, {\theta'}) \geqslant_b(\N_G(D), \N_M(D), \theta_0')$$ and 
$$(G, M, \htheta) \geqslant_c(\N_G(D), \N_M(D), \widehat \theta')$$ that are given by $(\cP,\cP')$ and $(\cP^*,{\cP'}^*)$ respectively.
In fact, it suffices to show that
for any $M\le J\le G$ and any projective representation $\mathcal Q$ of $J/M\cong \N_J(D)/\N_M(D)$ whose factor set is inverse to the one of $\cP|_{J}$ such that $\mathcal Q\otimes\cP|_J$ is an $\cO$-representation of $J$, one has that
$\mathcal Q$ can be chosen to be over $\cO$ by Lemma~\ref{lem-proj-cO}.
According to \cite[Lemma 2.6]{Sp13}, an isomorphism between ordinary-modular character triples with defect zero characters also preserves the partitioning of the (Brauer) characters into blocks.
So $(\cP^*,{\cP'}^*)$ gives $(G, M, {\htheta}) \geqslant_b(\N_G(D), \N_M(D), \widehat \theta')$.
\end{proof}

\subsection{Equivariant bijections}

Let $G\unlhd \tG$.
Now we recall the relationship ``covering" for weights between $G$ and $\tG$ defined in \cite[\S 2]{BS20}.
Let $(Q,\varphi)$ be a weight of $G$.
Let $M$ be a intermediate subgroup  $\N_G(Q)\le M\le \N_{\tG}(Q)_\varphi$ and $M/\N_G(Q)$ is an $\ell$-group.
We fix a defect group $\tQ/Q$ of the unique block of $M/Q$  covering $\bl_{\N_G(Q)/Q}(\vhi)$.
Then $\tQ\cap \N_G(Q)=Q$, $M=\N_G(Q)\tQ$ and 
$$\N_{M/Q}(\tQ/Q)=(\tQ/Q)\ti \C_{\N_G(Q)/Q}(\tQ/Q).$$
Note that $\pi_{\tQ/Q}(\varphi)\in\dz(\C_{\N_G(Q)/Q}(\tQ/Q))$.
We denote by $\bar\pi_{\tQ/Q}(\varphi)$ the associated character in $\Irr(\N_{M/Q}(\tQ/Q)/(\tQ/Q))$ which lifts to $\pi_{\tQ/Q}(\varphi)\times 1_{\tQ/Q}\in\Irr(\N_{M/Q}(\tQ/Q))$.
In addition,  $$\N_{M/Q}(\tQ/Q)/(\tQ/Q)\cong \N_G(\tQ)\tQ/\tQ$$ is normal in $\N_{\tG}(\tQ)/\tQ$.
Following Brough--Sp\"ath \cite{BS20}, we say that a weight $(\tQ,\tvhi)$ of $\tG$ \emph{covers} $(Q,\varphi)$ if $\tvhi\in\dz(\N_{\tG}(\tQ)/\tQ\mid \bar\pi_{\tQ/Q}(\varphi))$.

Let $(Q,\vhi)$ be a weight of $G$.
We denote by $\Alp(\tG\mid \overline{(Q,\vhi)})$ the set of those $\overline{(\tQ,\tvhi)}\in\Alp(\tG)$ such that $(\tQ,\tvhi)$ covers $(Q,\vhi)^g$ for some $g\in\tG$.

\begin{thm}\label{thm:cliff-equiva-bij}
Let $G$ and $\tG$ be normal subgroups of a finite group $\hG$ such that both $\hG/\tG$  and $\tG/G$ are abelian. Assume that $A$ is a finite group such that $\hG\unlhd A$ and $A$ stabilizes $G$ and $\tG$.
Suppose $B$ is a union of blocks of $G$ such that $B$ is $A$-invariant  and there exists a blockwise $A$-equivariant bijection $$\Omega:\IBr(B)\to\Alp(B)$$ such that $(A_\psi,G,\psi)\geqslant_b (\N_A(Q)_\vhi,\N_G(Q),\vhi^0)$ for every $\psi\in \IBr(B)$ and $\Omega(\psi)=\overline{(Q,\vhi)}$.
Assume further that for any $B$-weight
 $(Q,\vhi)$ of $G$,  if $(\tQ,\tvhi)\in\Alp^0(\tG)$ covers $(Q,\vhi)$ and 
$\vhi_0:=\pi_{\tQ/Q}(\vhi)\in\dz(\N_{G}(\tQ)/Q)$ be the DGN-correspondent,
then  $$(\N_A(Q)_\vhi,\N_G(Q)\tQ,\vhi')\geqslant_b(\N_A(\tQ)_\vhi,\N_G(\tQ)\tQ,(\vhi_0)')$$
for $\vhi'\in\IBr(\N_G(Q)\tQ\mid\vhi^0)$ and $(\vhi_0)'\in\IBr(\N_G(\tQ)\tQ\mid(\vhi_0)^0)$.

Let $\hcB$ (resp. $\tcB$) be the union of blocks of $\hG$ (resp. $\tG$) covering the blocks in $B$.
Assume further that for every $(Q,\vhi)\in\Alp^0(B)$, $\vhi$ extends to its stabilizer in $\N_{\tG}(Q)$, and for every $(\tQ,\tvhi)\in\Alp^0(\tcB)$, $\tvhi$ extends to its stabilizer in $\N_{\hG}(\tQ)$.
Then there exists a blockwise $A$-equivariant bijection $$\hOm:\IBr(\hcB)\to\Alp(\hcB).$$
\end{thm}

\begin{proof}	
We recall some technique from 	the proof of \cite[Thm.~3.16]{FLZ21}.
Let $\psi\in\IBr(B)$ and $\Omega(\psi)=\overline{(Q,\vhi)}$. 
Assume that $$(A_\psi,G,\psi)\geqslant_b (\N_A(Q)_\vhi,\N_G(Q),\vhi^0)$$ is given by $(\cP,\cP')$.
	We denote by $\sigma^{(\psi)}$ the isomorphism between the modular character triples  $(A_\psi,G,\psi)$ and $(\N_A(Q)_\vhi,\N_G(Q),\vhi^0)$ given by $(\cP,\cP')$.
	By an analogous argument as  in the proof of \cite[Prop.~4.7]{NS14},  we assume that  $\sigma^{(\psi)}$ satisfies that
	\begin{equation*}\label{equ:sigma-equiva}
		\sigma^{(\psi)}_J(\zeta)^x=\sigma^{(\psi^x)}_{J^x}(\zeta^x)\ \textrm{for every}\ G\le J\le A_\psi, \psi\in\IBr(B), \zeta\in\IBr(J\mid\psi), x\in \N_A(Q).
	\end{equation*}
By \cite[Lemma 3.3]{FLZ21} the bijection
	$$\sigma_{\tG_\psi}^{(\psi)}:\IBr(\tG_\psi\mid\psi)\to\IBr(\N_{\tG}(Q)_\vhi\mid \vhi^0)$$  satisfies that for every $\tpsi'\in\IBr(\tG_\psi\mid\psi)$ and $\widetilde{{\vhi'}^0}=\sigma_{\tG_\psi}^{(\psi)}(\tpsi')$ that
	$$((A_\psi)_{\widetilde \psi'},\tG_\psi,\tpsi')\geqslant_b ((\N_A(Q)_\vhi)_{\widetilde{{\vhi'}^0}},\N_{\tG}(Q)_\vhi,\widetilde{{\vhi'}^0}).$$

	For $\tpsi\in\IBr(\tcB)$, let $\psi\in\IBr(G\mid \tpsi)$ and $\tpsi'\in\IBr(\tG_{\psi}\mid\psi)$ such that $\Ind^{\tG}_{\tG_\psi} \tpsi'=\tpsi$.
	 The Brauer character $\tvhi':=\Ind^{\N_{\tG}(Q)}_{\N_{\tG}(Q)_\vhi}\sigma^{(\psi)}_{\tG_{\psi}}(\tpsi')$ is well-defined.
	Let $\mathcal T$ be a fixed complete set of representatives of $\tG$-orbits in $\IBr(B)$.
	Then  for every $\psi\in\mathcal T$ and $\Omega(\psi)=\overline{(Q,\vhi)}$, as in the proof of \cite[Thm.~3.15]{FLZ21}, 
	we can define a bijection $$\Pi_{\psi}:\ \IBr(\tG\mid \psi)\to \IBr(\N_{\tG}(Q)\mid \vhi^0),\quad
	\tpsi\mapsto \tvhi'$$ satisfies that
 $\N_{A}(Q)_{\tpsi}=N_{A}(Q)_{\Pi_\psi(\tpsi)}$ 
and	$$(A_{\tpsi},\tG,\tpsi)\geqslant_b (\N_{A}(Q)_{\tvhi'},\N_{\tG}(Q),\tvhi').$$
Together with $\Xi_\psi$ from \cite[Thm.~3.13]{FLZ21}, we have obtained a bijection between $\IBr(\tG\mid \psi)$ and $\Alp(\tG\mid \Omega(\psi))$.

We give more information of $\Xi_\psi$.
For every $(Q,\vhi)\in\Alp^0(B)$ and $(\tQ,\tvhi)\in\Alp^0(\tcB)$ such that $(\tQ,\tvhi)$ covers $(Q,\vhi)$,
by the assumption, one has
$$(\N_A(Q)_\vhi,\N_G(Q)\tQ,\vhi')\geqslant_b(\N_A(\tQ)_\vhi,\N_G(\tQ)\tQ,(\vhi_0)')$$
where $\vhi'\in\IBr(\N_G(Q)\tQ\mid\vhi^0)$, $(\vhi_0)'\in\IBr(\N_G(\tQ)\tQ\mid(\vhi_0)^0)$ and
$\vhi_0:=\pi_{\tQ/Q}(\vhi)\in\dz(\N_{G}(\tQ)/Q)$ is the DGN-correspondent.
By Lemma 3.3 and Proposition 3.4 of \cite{FLZ21}, we can deduce that 
there is a bijection $$\Pi'_\vhi:\IBr(\N_{\tG}(Q)\mid \vhi')\to \IBr(\N_{\tG}(\tQ)\mid (\vhi_0)')$$
such that $(\N_A(Q)_{\tvhi},\N_{\tG}(Q),\tvhi)\geqslant_b(\N_A(\tQ)_{\tvhi},\N_{\tG}(\tQ),\Pi_\vhi'(\tvhi))$.

Therefore, when $\psi$ runs through $\mathcal T$, by Lemma~\ref{lem:equ-transitive},
one has a blockwise $A$-equivariant bijection $\wOm:\IBr(\tcB)\to\Alp(\tcB)$ such that for every $\tpsi\in\IBr(\tcB)$,
	$\overline{(\tQ,\tvhi)}=\wOm(\tpsi)$,
$$(A_{\tpsi},\tG,\tpsi)\geqslant_b (\N_A(\tQ)_{\tvhi},\N_{\tG}(\tQ),\tvhi^0).$$	

Thus by the arguments above (similar as \cite[Thm.~3.15]{FLZ21}),
there exists a blockwise $A$-equivariant bijection $\hOm:\IBr(\hcB)\to\Alp(\hcB).$
\end{proof}

\section{Quasi-isolated blocks and Jordan decomposition of weights}
\label{sec:Jor-wei}

We establish the Jordan decomposition for weights of groups of type $\mathsf A$.
Since the local subgroups of a connected reductive group may be disconnected, we start from the algebraic groups which are possibly disconnected.
Let $p$ be a prime different from $\ell$, $\F_p$ the field of $p$ elements, and $\overline{\F}_p$ the algebraic closure of $\F_p$.
Recall that $\Lambda \in\{\cO,k\}$

\subsection{Reductive groups}
We recall  some elementary results on disconnected reductive groups  as follows, which can be found in \cite{DM94} and \cite[\S 4.8]{GM20}.

Let $\bG$ be a (possibly disconnected) reductive algebraic group defined over $\overline{\F}_p$.
Denote by $\bG^\circ$ the connected component of $\bG$ containing the identity.
A \emph{parabolic subgroup} of $\bG$ is any closed subgroup $\bP\le\bG$ which contains a Borel subgroup of $\bG^\circ$.
Then $\bP^\circ$ is a parabolic subgroup of $\bG^\circ$.
Let $\bU$ be the unipotent radical of $\bP^\circ$ and let $\bP^\circ=\bU\rtimes \bL^\circ$ be a Levi decomposition of $\bP^\circ$.
Then $\bP=\bU\rtimes \bL$ with $\bL=\N_{\bG}(\bL^\circ)$.
We call $\bL$ a \emph{Levi subgroup} of $\bG$.

Let $Q$ be a finite solvable $p'$-subgroup of $\bG$.
By \cite[\S 3]{BDR17}, $\N_{\bG}(Q)$ and $\C_{\bG}(Q)$ are reductive groups.
Let $\bP$ be a parabolic subgroup of $\bG$ and let $\bP=\bU\rtimes \bL$ be a Levi decomposition such that $Q\subseteq \bL$.
Then $\N_{\bP}(Q)$ is a parabolic subgroup of $\N_{\bG}(Q)$ with unipotent radical $\C_{\bU}(Q)$ and Levi decomposition $\N_{\bP}(Q)=\C_{\bU}(Q)\rtimes \N_{\bL}(Q)$.
Similarly, $\C_{\bP}(Q)$ is a parabolic subgroup of $\C_{\bG}(Q)$ with unipotent radical $\C_{\bU}(Q)$ and Levi decomposition $\C_{\bP}(Q)=\C_{\bU}(Q)\rtimes \C_{\bL}(Q)$.

Let $F:\bG\to\bG$ be a Steinberg endomorphism such that a power of $F$ is a Frobenius endomorphism defining a rational structure over a finite field $\F_q$ of $q$ elements, where $q$ is a power of $p$.
Let $\bL$ be an $F$-stable subgroup of $\bG$ and $\bP$ a parabolic subgroup of $\bG$ containing $\bL$ so that $\bP=\bU\rtimes \bL$ is the Levi decomposition.
The Deligne--Lusztig varieties 
$$\bY_{\bU}^{\bG}=\{ g\bU\in\bG/\bU\mid g^{-1}F(g)\in \bU\cdot F(\bU)  \}.$$
were introduced in \cite{DL76}, which has a left $\bG^F$- and right $\bL^F$- action.

\subsection{Bonnaf\'e--Dat--Rouquier Morita equivalence}

Now we assume further that $\bG$ is connected.
Let $(\bG^*,F^*)$ be in duality with $(\bG,F)$, where $F^*$ is a Steinberg endomorphism of $\bG^*$.
We also write it $F$ for short.
We recall that an element $t$ in $\bG^*$ is called \emph{quasi-isolated} if the centralizer $\C_{\bG^*}(t)$ is not contained in a proper Levi subgroup of $\bG^*$, while $t$ is called \emph{isolated} in $\bG^*$ if the connected centralizer $\C^\circ_{\bG^*}(t)$ is not contained in a proper Levi subgroup of $\bG^*$.
Following Definition 3.1 of \cite{Ru20b}, we call a semisimple element $t\in{\bG^*}^F$ is \emph{strictly quasi-isolated} if $\C_{{\bG^*}^F}(t)\C^\circ_{\bG^*}(t)$ is not contained in a proper Levi subgroup of $\bG^*$.

Let $s$ be a semisimple $\ell'$-element of ${\bG^*}^F$.
By \cite{BM89}, 
there exists a central idempotent $e_s^{\bG^F}$ of $\Lambda \bG^F$ which is uniquely determined by the ${\bG^*}^F$-conjugacy class containing $s$.
If  $b$ is a block idempotent of $\Z(\Lambda \bG^Fe_s^{\bG^F})$, then we say the block $b$ is \emph{in $s$}.

Assume that $s$ is not isolated let $\bL^*$ be an $F$-stable Levi subgroup of $\bG^*$ containing $\C_{\bG^*}^\circ(s)$.
Assume further that $\bL^*\C_{\bG^*}(s)^F=\C_{\bG^*}(s)^F\bL^*$.
Denote by $\bL$ an $F$-stable Levi subgroup of $\bG$ in duality with $\bL^*$.
We let $\bN^*=\C_{\bG^*}(s)^F\bL^*$ and
set $\bN$ to be the subgroup of $\N_{\bG}(\bL)$ such that $\bN/\bL\cong\bN^*/\bL^*$ under the canonical isomorphism between $\N_{\bG}(\bL)/\bL$ and $\N_{\bG^*}(\bL^*)/\bL^*$ induced by duality.
Recall that $\bG\hookrightarrow\tbG$ is a regular embedding.
Let $\tbN=\bN \Z(\tbG)$.

\begin{thm}[Bonnaf\'e--Dat--Rouquier]\label{thm:BDR-equ}
	Suppose that there exists a $\Lambda ((\bG^F\ti(\bN^F)^{\opp})\Delta \tbN^F)$-module $M'$ extending the
	$\Lambda ((\bG^F\ti(\bL^F)^{\opp})\Delta \tbL^F)$-module $H_c^{\dim}(\bY^{\bG}_{\bU},\Lambda)e_s^{{\bL}^F}$. 
	Then there exists a complex $\mathcal C$ of $\cO \bG^F$-$\cO\N^F$-bimodules extending $\G\Ga_c(\bY_{\bU}^\bG,\cO)^{\mathrm{red}}e_s^{\bL^F}$ so that the complex $\mathcal C$ induces a splendid Rickard equivalence between $\cO \bG^Fe_s^{\bG^F}$ and $\cO \bN^Fe_s^{\bL^F}$ and the bimodule $H^{\dim(\bY_{\bU}^\bG)}(\mathcal C)$ induces a Morita equivalence between 
	$\cO \bG^Fe_s^{\bG^F}$ and $\cO \bN^Fe_s^{\bL^F}$.
\end{thm}

\begin{proof}
	This is proved in \cite[\S7]{BDR17}, where 
	in the proof of \cite[Thm.~7.5]{BDR17} use the assumption that $M'$ extends $H_c^{\dim}(\bY^{\bG}_{\bU},\Lambda)e_s^{{\bL}^F}$ instead of
	\cite[Prop.~7.3]{BDR17}.
	See also \cite{Ru20}.
\end{proof}

If $\bN^F/\bL^F$ is cyclic, then  the extendibility property of $H_c^{\dim}(\bY^{\bG}_{\bU},\Lambda)e_s^{{\bL}^F}$ in Theorem~\ref{thm:BDR-equ} always holds (see for instance \cite[Lemma~10.2.13]{Ro98}).

\subsection{Groups of type $\mathsf A_{n-1}$}

From now on we consider groups of type $\mathsf A$ and we first recall some notation as in \cite{Ru21}.
Let $\bG=\SL_n(\overline\F_p)$ and $\tbG=\GL_n(\overline\F_p)$ and let $\eps\in\{\pm 1\}$.
For any power $q_0$ of $p$ and let $F_{q_0}:\tbG\to\tbG$ be the standard Frobenius endomorphism $(a_{ij})\mapsto (a_{ij}^{q_0})$.
Suppose $\gamma':\tbG\to\tbG$ is the graph automorphism given by transpose inversion $g\mapsto g^{-\mathrm{tr}}$.
We fix a power $q$ of $p$.
Let $F'=F_q(\ga')^{\frac{1-\eps}{2}}$.

Let $\bT_0$ be the torus of diagonal matrices and $\bB_0$ be the Borel subgroup of unitriangular matrices.
Note that $\bB_0$ is not $\ga'$-stable and thus $\bT_0$ is not split with respect to $F'$.
As in  \cite[\S 3.2]{CS17} and \cite[\S 5.2]{Ru21}, we define $\ga=n_0\ga'$, where $n_0\in\N_{\bG}(\bT_0)$ is the matrix with entry $(-1)^{l+1}$ at position $(l,n+1-l)$ for $1\le l\le n$ and 0 elsewhere. 
Then $\bB_0$ is $\langle\ga\rangle$-stable. 
For $\eps\in\{\pm 1\}$ we let $F:\tbG\to\tbG$ be the Frobenius endomorphism given by $F=F_q\ga^{\frac{1-\eps}{2}}$.
Then $\bT_0$ is maximal split.
By Lang--Steinberg theorem (see for example \cite[Thm.~1.4.8]{GM20}), $\bG^F$ is isomorphic to ${\bG}^{F'}$ via conjugation by some element in $\bG$.
In addition, the containment $\bG\le\tbG$ is a regular embedding (see \cite[Definition 1.7.1]{GM20}).
Let $G=\bG^F$ and $\tG=\tbG^F$.

We consider the dual groups.
Now $\bG^*$ is adjoint of type $\mathsf A$, then there is a surjective morphism $\pi:\bG\to\bG^*$.
If $f:\bG\to\bG$ is a bijective morphism such that $f(\bT_0)=\bT_0$, then we denote by $f^*:\bG^*\to\bG^*$ the unique morphism satisfying $\pi\circ f=f^*\circ\pi$.
Moreover, if $f$ is a Frobenius endomorphism then the pair $(\bG^*,f^*)$ is in duality with $(\bG,f)$.
For a Levi subgroup $\bL$ of $\bG$, we define $\bL^*=\pi(\bL)$.

For simplicity of notation, we write $F$ instead of $F^*$ for the dual groups.
Let $s$ be a strictly quasi-isolated semisimple $\ell'$-element of ${\bG^*}^F$ such that $s\ne 1$.
We denote by $\bL$ the Levi subgroup of $\bG$ dual to
$\bL^*=\C^\circ_{\bG^*}(s)$, that is, $\bL^*=\pi(\bL)$.
Then $\bL$ is an $F$-stable Levi subgroup of $\bG$.
As in \cite[\S 6]{Ru21}, there exists a proper $F$-stable Levi subgroup $\bL'$ of $\bG$ containing $\bL$ such that $\C_{{\bG^*}^F}(s)(\bL')^*=(\bL')^*\C_{{\bG^*}^F}(s)$ and $N'/L'$ is cyclic of prime order, where $L'={\bL'}^F$ and $N'$ is the common stabilizer of $\bL'$ and $e_s^{L'}$ in $\bG^F$.
In addition, $N'/L'$ is naturally isomorphic to a subquotient of $N/L$.
Obviously, $N'/L'$ is of $\ell'$-order.
By Theorem~\ref{thm:BDR-equ}, $\Lambda N'e_s^{L'}$ and $\Lambda Ge_s^G$ are Morita equivalent.
Let $\tN'=\N_{\tG}(\bL',e_s^{L'})$.

From now on until the end of this section, we always assume that $\ell\nmid (q-\eps)$. 

\subsection{Automorphisms}

Let $\cB$ be the subgroup of $\Aut(\tbG^F)$ generated by $F_p$ and $\ga$.
Then by \cite[Thm.~2.5.1]{GLS98}, $\tG\rtimes\cB$ induces all automorphisms of $G$.
Denote by $\cB_{e_s^{G}}$ the stabilizer of $e_s^{G}$ in $\cB$.
If $(\bG,F)$ is untwisted, then $\cB_{e_s^{G}}=\langle \ga_0,\phi\rangle$, where $\phi:\bG\to\bG$ is a power of $F_p$ or $F_p\ga$ (a Frobenius endomorphism) such that $\phi^r=F$ for some $r$, and $\ga_0\in\{\mathrm{id}_{\bG},\ga\}$.
If $(\bG,F)$ is twisted, then $\cB=\langle F_p\rangle$ and $\cB_{e_s^{G}}=\langle \phi\rangle$, where $\phi$ is a power of $F_p$ (a Frobenius endomorphism) such that $\phi^r=F\ga_0$ with $\ga_0\in\{\mathrm{id}_{\bG},\ga\}$.

We recall the subgroup $\cA$ of $\Aut(\tG)$ from Definition 7.6 of \cite{Ru21}.
If $\bL'$ is not $1$-split, then
 we define $\cA=\cB_{e_s^{G}}$.
If $(\bG,F)$ is untwisted and $\bL'$ is 1-split, then according to \cite[Lemma 6.12]{Ru21}, by possibly replacing $s$ by a $(\bG^*)^F$-conjugate, there exists a Frobenius endomorphism $F_0:\tbG\to\tbG$ commuting with $\ga_0$ such that $F_0^r=F$, and $F_0$ stabilizer $\bL'$ and $e_s^{L'}$.
In this case, we define $\cA:=\langle F_0,\ga_0\rangle$. 
If $(\bG,F)$ is twisted and $\bL'$ is 1-split, then according to \cite[Lemma 6.14]{Ru21},
by possibly replacing $s$ by a $(\bG^*)^F$-conjugate, there exists a Frobenius endomorphism $F_0:\tbG\to\tbG$ commuting with $\ga_0$ such that $F_0^r=F\ga_0$, and $F_0$ stabilizer $\bL'$ and $e_s^{L'}$.
In this situation, we define $\cA:=\langle F_0\rangle$. 

In all cases, $\cA$ is an abelian group and
the group $\tG\cA$ induces the stabilizer of $e_s^G$ in $\Out(G)$.
We set $\cN=\N_{\tG\cA}(\bL',e_s^{L'})$.
Since $N'/L'$ is of $\ell'$-order, we know that the Hall $\ell'$-subgroup of $\cN/\tL'$ is normal.
Denote by $\cN_{\ell'}/\tL'$ the unique Hall $\ell'$-subgroup of $\cN/\tL'$.

\subsection{Equivariant bijections}\label{subsec:equ-bijection}

Let $\bP'$ be the parabolic subgroup of $\bG$ whose Levi complement is $\bL'$ as in \cite[Cor.~6.9]{Ru21}.
Let $\bP'=\bU'\rtimes \bL'$ be the Levi decomposition.
By \cite[Thm.~7.7]{Ru21}, 
the $\Lambda ((G\ti {L'}^{\opp})\Delta \tL')$-module $H^{\dim}(\bY_{\bU'}^{\bG},\Lambda)e_s^{L'}$ extends to a $\Lambda ((G\ti {L'}^{\opp})\Delta \cN_{\ell'})$-module, which is $\Delta \cN$-stable and
denoted by $\hM$.
Let $M'$ be the restriction of $\hM$ to $(G\ti {N'}^{\opp})\Delta \tN'$.

By Theorem ~\ref{thm:BDR-equ}, there exists a complex $\mathcal C'$ of $\Lambda (G\ti (N')^{\opp})$-modules extending $\G\Ga_c(\bY_{\bU'}^{\bG},\Lambda)^{\mathrm{red}} e_s^{L'}$
such that $H^{\dim(\bY_{\bU'}^{\bG})}(\mathcal C')\simeq M'$ and $\mathcal C'$ induces a splendid Rickard equivalence between $\Lambda N'e_s^{L'}$ and $\Lambda Ge_s^G$.

\begin{prop}\label{prop:-bij-Brauer}
	There is an $\cN$-equivariant bijection between $\IBr(N',e_s^{L'})$ and $\IBr(G,e_s^{G})$.
\end{prop}

\begin{proof}
	This follows from the fact that the bimodule $M'$ induces a Morita equivalence between $\Lambda N'e_s^{L'}$ and $\Lambda Ge_s^G$.
\end{proof}

Let $c$ be a block in $\Z(\Lambda N' e_s^{L'})$ and $b$ the block of $G$ corresponding to $c$ via $\mathcal C'$.
If  $(Q,c_Q)$ is a $c$-Brauer pair, then by \cite{Ha99}, there exists a unique $b$-Brauer pair $(Q,b_Q)$ such that the complex $b_Q\Br_{\Delta Q}(\mathcal C')c_Q$ induces a Rickard equivalence between $k\C_G(Q)b_Q$ and $k\C_{N'}(Q)c_Q$.
By \cite[Thm.~19.7]{Pu99}, splendid Rickard equivalence between blocks induces an equivalence between Brauer categories.
Then similar with Lemma 4.3 and 4.6 of \cite{FLZ21}, we have

\begin{lem}\label{lem:equ-brauer-cat}
	$(Q,c_Q)\mapsto (Q,b_Q)$ is an $\cN_c$-equivariant injective map from $\bBr(L',c)$ to $\bBr(G,b)$. 
\end{lem}

For Brauer pairs
$(Q,c_Q)$ and $(Q,b_Q)$ as above, we define $C_Q=\Tr_{\N_{N'}(Q,c_Q)}^{\N_{N'}(Q)}(c_Q)$ and 
$B_Q=\Tr_{\N_{G}(Q,b_Q)}^{\N_{G}(Q)}(b_Q)$.

\begin{lem}\label{lem:equ-nor-pair}
	$(Q,C_Q)\mapsto (Q,B_Q)$ is an $\cN_c$-equivariant injective map from $\bNBr(L',c)$ to $\bNBr(G,b)$. 	
\end{lem}

\begin{proof}
	By Lemma~\ref{lem:equ-brauer-cat}, the map 	$(Q,C_Q)\mapsto (Q,B_Q)$ is well-defined.
	Thus the proof of this lemma is similar as \cite[Lemma 4.8]{FLZ21}.
\end{proof}

Now let $(Q,C_Q)\in \bNBr(L',c)$, and let $(Q,B_Q)$ be the corresponding pair in $\bNBr(G,b)$ given as above.
By \cite[Thm.~5.2]{Ri96}, there is a unique complex $\mathcal C_Q'$ of $\ell$-permutation $\Lambda((\C_G(Q)\ti \C_{N'}(Q)^{\opp})\Delta \N_{N'}(Q))$-modules, which is a lift of the complex $\Br_{\Delta Q}(\mathcal C')$ of $\ell$-permutation modules from $k$ to $\Lambda$.
Let $$M_Q'=\Ind^{\N_G(Q)\ti \N_{N'}(Q)^{\opp}}_{(\C_G(Q)\ti \C_{N'}(Q)^{\opp})\Delta \N_{N'}(Q)}H^{\dim(\bY_{\C_{\bU'}(Q)}^{\C_{\bG}(Q)})}(\mathcal C_Q')\mathrm{br}_Q(e_s^{L'}).$$
Then $M_Q'$ is an extension of $H_c^{\dim}(\bY^{\N_{\bG}(Q)}_{\C_{\bU'}(Q)},\Lambda)\mathrm{br}_Q(e_s^{L'})$.
We have the following Morita equivalence for local situation.

\begin{thm}\label{thm:local-Ruh}
	The bimodule $M'_QC_Q$ induces a Morita equivalence between $\Lambda \N_{N'}(Q)C_Q$ and $\Lambda \N_G(Q)B_Q$.
\end{thm}

\begin{proof}
	This is similar as the proof of \cite[Thm.~3.10]{Ru20a}.
\end{proof}

Next, we consider the Morita equivalence between blocks of quotient groups.
Let $\overline{\N_{N'}(Q)}=\N_{N'}(Q)/Q$ and $\overline{\N_{G}(Q)}=\N_{G}(Q)/Q$, and let
$$\overline{M'_QC_Q}=\Lambda \overline{\N_{N'}(Q)}\bar C_Q\otimes_{\Lambda \N_{N'}(Q)} M'_QC_Q\otimes_{\Lambda \N_G(Q)} \Lambda \overline{\N_G(Q)}\bar B_Q.$$
Here $\bar C_Q$ and $\bar B_Q$ are the images of $C_Q$ and $B_Q$ through the canonical morphisms $\Lambda \N_{N'}(Q)\to \Lambda \overline{\N_{N'}(Q)}$
and $\Lambda \N_{G}(Q)\to \Lambda \overline{\N_{G}(Q)}$.

\begin{thm}\label{thm:local-quotient}
	The bimodule $\overline{M'_QC_Q}$ induces a Morita equivalence between $\Lambda \overline{\N_{N'}(Q)}\bar C_Q$ and $\Lambda \overline{\N_G(Q)}\bar B_Q$.
\end{thm}

\begin{proof}
	For simplicity we denote $d_Q=\dim(\bY_{\C_{\bU'}(Q)}^{\C_{\bG}(Q)})$.
	Let $\mathcal C_Q''=\Ind^{(\N_G(Q)\ti \N_{N'}(Q)^{\opp})\Delta \N_{\tN'}(Q)}_{(\C_G(Q)\ti \C_{N'}(Q)^{\opp})\Delta \N_{\tN'}(Q)}\mathcal C_Q'$.
	Then 
	$\mathcal C_Q''$ induces a Rickard equivalence between $\Lambda \N_{N'}(Q)C_Q$ and $\Lambda \N_G(Q)B_Q$
	and the 
	bimodule $H^{d_Q}(\mathcal C_Q'')\simeq M'_QC_Q$ induces a Morita equivalence between $\Lambda \N_{N'}(Q)C_Q$ and $\Lambda \N_G(Q)B_Q$.
	Let
	$$\overline{\mathcal C_Q''}=\Lambda \overline{\N_{N'}(Q)}\bar C_Q\otimes_{\Lambda \N_{N'}(Q)} \mathcal C_Q''\otimes_{\Lambda \N_G(Q)} \Lambda \overline{\N_G(Q)}\bar B_Q.$$
	
	As \cite[Lemma 4.12]{FLZ21}, we can show that $H^i(\overline{\mathcal C''_Q})\ne 0$ if and only if $i=d_Q$.	
	In addition, $H^{d_Q}(\overline{\mathcal C''_Q})\simeq \overline{H^{d_Q}(\mathcal C''_Q)}$.
	Then the proof of \cite[Thm.~4.13]{FLZ21} applies here and we only give a sketch.
	
	First, 	if $\Lambda=k$, then $\mathcal C_Q'=\Br_{\Delta Q}(\mathcal C')$ and the complex $\overline{\mathcal C_Q''}$ induces a Rickard equivalence between 
	$k \overline{\N_{N'}(Q)}\bar C_Q$ and $k \overline{\N_G(Q)}\bar B_Q$.
	Then by	 \cite[Thm.~5.2]{Ri96},  when lifting to $\Lambda$, the complex $\overline{\mathcal C_Q''}$ induces a Rickard equivalence between 
	$\Lambda \overline{\N_{N'}(Q)}\bar C_Q$ and $\Lambda \overline{\N_G(Q)}\bar B_Q$.
	Since	$\overline{\mathcal C_Q''}$ is concentrated in a single degree,  the bimodule
	$\overline{H^{d_Q}(\mathcal C''_Q)}\simeq H^{d_Q}(\overline{\mathcal C''_Q})$
	induces a Morita equivalence between 
	$\Lambda \overline{\N_{N'}(Q)}\bar C_Q$ and $\Lambda \overline{\N_G(Q)}\bar B_Q$.
\end{proof}

Now we establish an equivariant bijection between weights.

\begin{prop}\label{prop:-bij-wei}
	There is an $\cN$-equivariant bijection between $\Alp(N',e_s^{L'})$ and $\Alp(G,e_s^{G})$.	
\end{prop}

\begin{proof}
	By Theorem~\ref{thm:local-Ruh}, the bimodule $M_Q'$ induces a bijection
	$$\Irr(\N_{N'}(Q),C_Q)\to\Irr(\N_{G}(Q),B_Q).$$
	Then by Theorem~\ref{thm:local-quotient},
	this gives a bijection between $$\dz(\N_{N'}(Q)/Q)\cap \Irr(\N_{N'}(Q),C_Q)\to\dz(\N_{G}(Q)/Q)\cap\Irr(\N_{G}(Q),B_Q).$$
	Thus using Lemma~\ref{lem:equ-nor-pair}, 
	this gives a
	bijection between $\Alp(N',e_s^{L'})$ and $\Alp(G,e_s^{G})$,
	which is similar as in the proof of \cite[Thm.~4.15]{FLZ21}.
	
	By Lemma~\ref{lem:equ-nor-pair}, in order to complete the proof, it suffices to show that  the bijection $$\Irr(\N_{N'}(Q),C_Q)\to\Irr(\N_{G}(Q),B_Q)$$ induced by the bimodule  $M_Q'$ is $\N_{\cN}(Q,C_Q)$-equivariant.
	We shall have established this if we prove the following: the bimodule
	$M_Q'$ extends to a   $\Lambda ((\N_G(Q)\ti \N_{N'}(Q)^{\opp})\Delta \N_{\cN_{\ell'}}(Q))$-module, which is $\Delta\N_{\cN}(Q)$-stable.

	We first assume that $\bL'$ is 1-split.
	Then  the proof of \cite[Lemma 8.3]{Ru21} applies and we recall as follows.
	In this case  $\cC'\simeq M'$ and $H^0(\Br_{\Delta Q}(\cC'))\simeq \Br_{\Delta Q}(M')$.
	In this way, $\Br_{\Delta Q}(\hM)$ is a $k((\C_G(Q)\ti\C_{L'}(Q)^{\opp})\Delta \N_{\cN_{\ell'}}(Q))$-module extending $\Br_{\Delta Q}(M')$ since $\hM$ is a $\Lambda ((G\ti {L'}^{\opp})\Delta \cN_{\ell'})$-module extending $M'$.
	In addition, $\Br_{\Delta Q}(\hM)$  is an $\ell$-permutation module.
	Thus by \cite[Thm.~5.11.2]{Lin18}, there exists a unique $\ell$-permutation module $\hM_Q$ which lifts $\Br_{\Delta Q}(\hM)$ to $\Lambda$ and it follows from the uniqueness of this lift that $\hM_Q$ is $\N_{\cN}(Q)$-stable.
	So $H^0(\cC_Q')\mathrm{br}_Q(e_s^{L'})$ extends to $(\C_G(Q)\ti\C_{N'}(Q)^{\opp})\Delta \N_{\cN_{\ell'}}(Q)$ which is $\Delta \N_{\cN}(Q)$-stable.
	By applying the induction functor, the bimodule
	$M_Q'$ extends to a   $\Lambda ((\N_G(Q)\ti \N_{N'}(Q)^{\opp})\Delta \N_{\cN_{\ell'}}(Q))$-module, which is $\Delta\N_{\cN}(Q)$-stable.

	Now we assume that $\bL'$ is not 1-split. 
	By Lemma 6.11 and 6.13 of \cite{Ru21},  $\cN/\tL'$ is abelian.
	So according \cite[lemma 8.2]{Ru21} 
	there exists a $\Delta\N_{\cN}(Q)$-stable $\Lambda ((\N_G(Q)\ti \N_{N'}(Q)^{\opp})\Delta \N_{\cN_{\ell'}}(Q))$-module $\hM_Q'$ extending
	the bimodule
	$H_c^{\dim}(\bY^{\N_{\bG}(Q)}_{\C_{\bU'}(Q)},\Lambda)\mathrm{br}_Q(e_s^{L'})$.
	Let $M_Q''$ be the restriction of $\hM_Q'$ to $\N_G(Q)\ti \N_{N'}(Q)^{\opp}$.
	Note that  the bimodules $M_Q'$ and $M_Q''$ are  extensions of the  $\Lambda ((\N_G(Q)\ti \N_{L'}(Q)^{\opp})$-module $H_c^{\dim}(\bY^{\N_{\bG}(Q)}_{\C_{\bU'}(Q)},\Lambda)\mathrm{br}_Q(e_s^{L'})$.
	Since $N'/L'$ is of $\ell'$-order, by Lemma~\ref{lem:ext-abel}, there exists a $\Lambda (\N_{N'}(Q)/\N_{L'}(Q))$-module $V$ such that $\dim_\Lambda V=1$ and $M_Q'=M_Q''\otimes V$.
	Let $\lambda\in\Irr(\N_{N'}(Q)/\N_{L'}(Q))$ corresponding to $V$.
	
	We may regard $\lambda$ as a linear character of $\tN'/\tL'\cong N'/L'$.
Then there is an extension  $\lambda'$ of $\lambda$ to $\cN/\tL'$ since $\cN/\tL'$ is abelian.
	Regard $\lambda'$ as a character of $(\N_G(Q)\ti \N_{L'}(Q)^{\opp})\Delta \N_{\cN}(Q)/(\N_G(Q)\ti \N_{L'}(Q)^{\opp})\Delta \N_{\tL'}(Q)$.
	Let $V_Q'$ be the $\Lambda ((\N_G(Q)\ti \N_{N'}(Q)^{\opp})\Delta \N_{\cN_{\ell'}}(Q))$-module corresponding to $\Res^{(\N_G(Q)\ti \N_{L'}(Q)^{\opp})\Delta \N_{\cN}(Q)}_{(\N_G(Q)\ti \N_{L'}(Q)^{\opp})\Delta \N_{\cN_{\ell'}}(Q)}\lambda'$.
	Then $\hM_Q'\otimes V_Q'$ is an extension of $M_Q'$, which is $\Delta\N_{\cN}(Q)$-stable.
	This completes the proof.
\end{proof}

\section{Reduce to isolated blocks}
\label{sec:red-iso}

Recall that $G=\SL_n(\eps q)$ and $\tG=\GL_n(\eps q)$.
The blocks of $\tG$  were classified by Fong--Srinivasan \cite{FS82} and Brou\'e \cite{Br86}, while the weights of $\tG$ were obtained by Alperin--Fong \cite{AF90} and An \cite{An92,An93,An94}.
Using these, there exists a blockwise $\IBr(\tG/G)\rtimes \cB$-equivariant bijection between $\IBr(\tG)$ and $\Alp(\tG)$, see \cite{Fe19} and \cite{LZ18}.
In  addition, the weights of $G$ were classified 
and the inductive condition of the non-blockwise Alperin weight conjecture for simple groups of type~$\mathsf A$ was established for $G$ in \cite{FLZ20a}. 

We first give a criterion for the inductive BAW condition for groups of type $\mathsf A$.

\begin{prop}\label{prop:indu-condition}
	Let $s$ be a quasi-isolated semisimple $\ell'$-element of $(\bG^*)^F$.
	Keep the notation in \S~\ref{subsec:equ-bijection}.
	Assume that  the following hold.
	\begin{enumerate}[\rm(i)]
		\item There exists a blockwise $\Lin_{\ell'}(\tN'/N')\rtimes \cN$-equivariant bijection $$\tilde f:\IBr(\tN'\mid \IBr(N',e_s^{L'}))\to \Alp(\tN'\mid \Alp(N',e_s^{L'})).$$ 
		\item There exists a blockwise $\cN$-equivariant bijection $f:\IBr(N',e_s^{L'})\to\Alp(N',e_s^{L'})$ such that
		\begin{enumerate}[\rm(a)]
        \item $\tilde f(\IBr(\tN'\mid\psi))=\Alp(\tN'\mid f(\psi))$ for any $\psi\in\IBr(N,e_s^{L'})$, and
        \item for any $\psi\in\IBr(N',e_s^{L'})$, $\tpsi\in\IBr(\tN'\mid\psi)$ and 		$f(\psi)=\overline{(Q,\vhi)}$, $\tilde f(\tpsi)=\overline{(\tQ,\tvhi)}$	 the following hold:  $\bl(\hpsi)=\bl(\widehat\vhi)^{\tN'_\psi}$, where 	 $\hpsi\in\IBr(\tN'_\psi\mid\psi)$ is the Clifford correspondent of $\tpsi$ and	 $\widehat\vhi\in\Irr(\N_{\tN'}(Q)_\vhi\mid\vhi)$ is the Clifford correspondent of $\Delta_\vhi^{-1}(\tpsi)$. 
		\end{enumerate}
		\end{enumerate}
	Then any block  $b$  in $s$ is BAW-good.	
\end{prop}

\begin{proof}
	We will use the criterion for the inductive BAW condition given by Brough and Sp\"ath \cite{BS20} and see \cite[Thm.~3.18]{FLZ21}.
	As pointed in the proof of \cite[Prop.~5.2]{FLZ21}, condition    (i)  of \cite[Thm.~3.18]{FLZ21} holds and for the rest conditions we transfer from $\cB$ to $\cA$.
	By Theorem 7.1 and 8.1 of \cite{FLZ20a}, the condition   (iv) of \cite[Thm.~3.18]{FLZ21} holds.
	According to Proposition~\ref{prop:-bij-Brauer}, there is an $\cN$-equivariant bijection  $$\IBr(N',e_s^{L'})\to\IBr(G,e_s^{G}),$$ while
	by Proposition~\ref{prop:-bij-wei}, there is an $\cN$-equivariant bijection  $$\Alp(N',e_s^{L'})\to\Alp(G,e_s^{G}).$$
	Thus condition (ii) of \cite[Thm.~3.18]{FLZ21} holds.

	By \cite[Prop.~1.1]{BR06}, one has the canonical isomorphism
	$$\Ind^{\tG\ti (\tL')^{\opp}}_{(G\ti (L')^{\opp})\Delta \tL'}\G\Ga_c(\bY_{\bU'}^{\bG},\cO)e_s^{L'} \simeq \G\Ga_c(\bY_{\bU'}^{\tbG},\cO)e_s^{L'}$$
	in $\mathrm{Ho}^b(\cO(\tG\ti\tL^{\opp}))$.
	Let $\cC=\G\Ga_c(\bY_{\bU'}^{\bG},\cO)^{\mathrm{red}}e_s^{L'}$ and $\widetilde{\cC}=\Ind^{\tG\ti (\tL')^{\opp}}_{(G\ti (L')^{\opp})\Delta \tL'}(\cC)$.
	
	Recall that $\iota^*:\tbG^*\twoheadrightarrow \bG$ denotes canonical epimorphism.
	Let $J$ be a set of representatives of conjugacy classes of $\ell'$-elements $\tilde t\in(\tbG^*)^F$ such that $\iota(\tilde t)=s$.
	Then $J\subseteq {(\tbL')^*}^F$. Let $e=\sum_{\tilde t\in J}  e_{\tilde t}^{\tL'}$.
	By \cite[Lem.~7.4]{BDR17}, one has the idempotent decompositions $e_s^{G}=\sum_{\tilde t\in J} e_{\tilde t}^{\tG}$ and $e_s^{L'}=\sum_{n\in N'/L'}  n e n^{-1}$.
	Similar as the proof of \cite[Thm.~7.5]{BDR17}, the complex $\widetilde{\cC}'=\widetilde{\cC}e\otimes_{\Lambda L'}\Lambda N'$ induces a splendid Rickard equivalence between $\Lambda \tN'e_s^{L'}$ and $\Lambda \tG e_s^{G}$.
	The cohomology of $\widetilde{\cC}'$ is concentrated in degree $\dim(\bY_{\bU'}^{\bG})$.
	Let $\tM'=H^{\dim(\bY_{\bU'}^{\bG})}(\widetilde{\cC}')$.
	Then 
	$\tM'\simeq (H^{\dim}_c(\bY_{\bU'}^{\tbG},\Lambda) e_s^{L'}) e\otimes_{\Lambda L'}\Lambda N'$
	induces a Morita equivalence between $\Lambda \tN'e_s^{L'}$ and $\Lambda \tG e_s^{G}$.
	Moreover, $\tM'\simeq \Ind^{\tG\ti (\tN')^{\opp}}_{(G\ti (L')^{\opp})\Delta \tN'} (M')$.
	Then similar as in the proof of \cite[Prop.~5.2]{FLZ21}, condition (iii) of \cite[Thm.~3.18]{FLZ21} also holds,  using the methods of \cite[\S 1.8]{Ru21}.
\end{proof}	

Now we reduce the verification of the  inductive BAW condition for type $\mathsf A$ to the isolated blocks.

\begin{thm}\label{thm:red-isolated}
	Assume that all isolated $\ell$-blocks of quasi-simple groups of Lie type $\mathsf A$ defined over a field of characteristic different from $\ell$ are BAW-good.
	Then all $\ell$-blocks of  quasi-simple groups of Lie type $\mathsf A$ are BAW-good.
\end{thm}

\begin{proof}
	By \cite[Prop.~4.6]{FLZ20b}, the simple groups $\PSL_n(q)$ and $\PSU_n(q)$ are BAW-good if $n\le 7$, which implies that we only need to consider the non-exceptional Schur multiplier case. 
	Note that the simple groups $\PSL_3(4)$, $\PSU_4(3)$ and $\PSU_6(2)$ were dealt with in \cite{Du19,Du20}.	
	As above we let $G=\bG^F=\SL_n(\eps q)$.
	By \cite[Thm.~5.7]{FLZ21}, it suffices to show that the strictly quasi-isolated blocks of $\bG^F$ are BAW-good, since Assumption 5.3 of \cite{FLZ21} holds for $G$ by \cite[Thm.~4.1]{CS17} and the unitriangularity of decomposition matrices in \cite{KT09,De17}; see also \cite[Thm.~8.1]{FLZ20a}.
	
	We show this assertion by induction on the rank of $\bG$.	
	Let $b$ be a strictly quasi-isolated $\ell$-block of $G$, which is not isolated, and 
	let $s$ be a semisimple $\ell'$-element of ${\bG^*}^F$ such that $b$ is in $s$.
	Keep the notation of the previous section.	
	Let $c$ be the block of $N'$ which is Morita equivalent to $b$ as in the previous section.	
	Then we verify the inductive BAW condition by showing the conditions of Proposition~\ref{prop:indu-condition}.
	
	Let $L'_0=[\bL',\bL']^F$ and let $c_0$ be a block of $L_0'$ covered by $c$.	
	Note that $$L_0'=H_1\ti \cdots\ti H_t$$ where the finite groups $H_i$ are special linear or unitary groups.	
	Let $$c_0=c_{0,1}\otimes\cdots\otimes c_{0,t},$$ where $c_{0,i} $ is a block of $H_i$ for $1\le i\le t$.
	If $H_i$ is quasi-simple, then we know that $c_{0,i}$ is BAW-good	by our induction hypothesis. If $H_i$ is solvable, then there is an iBAW-bijection for $c_{0,i}$ by \cite{FLZ20b}; see also the proof of  \cite[Prop.~5.6]{FLZ21}.
	In this way,  similar	as the proof of  \cite[Prop.~5.6]{FLZ21}, there is an iBAW-bijection  $\IBr(c_0)\to \Alp(c_0)$.
	
	We use the following lemma, whose proof is postponed until after this one is complete.
	
	\begin{lem}\label{lem:ext-geq_b}
		Let $(Q_0,\vhi_0)$ be a weight of $L'_0$ and $(Q,\vhi)$ be a weight of $L'$ covering $(Q_0,\vhi_0)$. Let $\vhi_0'\in\dz(\N_{L'_0}(Q)/Q_0)$ be the DGN-correspondent of $\vhi_0$.
		Then $$(\N_{\cN}(Q_0)_{\hvhi},\N_{L_0'}(Q_0)Q,\hvhi)\geqslant_b (\N_{\cN}(Q)_{\hvhi'},\N_{L_0'}(Q)Q,\hvhi')$$
		where $\hvhi\in\IBr(\N_{L_0'}(Q_0)Q\mid(\vhi_0)^0)$ and $\hvhi'\in\IBr(\N_{L_0'}(Q)Q\mid (\vhi_0')^0)$.
	\end{lem}
	
	Thus by \cite[Lemma 3.12]{FLZ21} and Theorem~\ref{thm:cliff-equiva-bij}, this assertion holds by a similar argument as in the proof of \cite[Thm.~5.7]{FLZ21}.
\end{proof}

\begin{proof}[Proof of Lemma~\ref{lem:ext-geq_b}]
Recall that $L_0'=H_1^{t_i}\ti \cdots\ti H_u^{t_u}$, where $H_i$ is a special linear or unitary
group for every $i$ and $H_i$ and $H_j$ are not isomorphic if $i\ne j$.
Let $\tbL'=\Z(\bG)\bL'$.
Then $$\tL'=(\tbL')^F=\tH_1^{t_i}\ti \cdots\ti\tH_u^{t_u}$$ where $\tH_i$ is the general  linear or unitary group corresponding to $H_i$.
Let $\cB_i$ be the group generated by field and graph automorphisms of $\tH_i$.
Then by \cite[Thm.~2.5.1]{GLS98}, $\tH_i\cB_i$ induces all automorphisms of $H_i$.
So the group $A=\prod_{i=1}^{u} \tH_i\cB_i \wr\fS_{t_i}$ induces all automorphisms of $L_0'$.
By \cite[Thm.~3.5]{Sp17}, it suffices to show that
 $$(\N_{A}(Q_0)_{\hvhi},\N_{L_0'}(Q_0)Q,\hvhi)\geqslant_b (\N_{A}(Q)_{\hvhi'},\N_{L_0'}(Q)Q,\hvhi')$$
where $\hvhi\in\IBr(\N_{L_0'}(Q_0)Q\mid(\vhi_0)^0)$ and $\hvhi'\in\IBr(\N_{L_0'}(Q)Q\mid (\vhi_0')^0)$.

Let $Q_0=\prod_{i=1}^u\prod_{j=1}^{m_i} (Q_{0,i,j})^{t_{ij}}$, where $Q_{0,i,j}$ is a radical subgroup of $H_i$, 
$Q_{0,i,j_1}\ncong Q_{0,i,j_2}$ if  $j_1\ne j_2$,
and $\sum_{j=1}^{m_i}t_{ij}=t_i$.
Thus $$\N_{L_0'}(Q_0)=\prod_{i=1}^u\prod_{j=1}^{m_i}\N_{H_i}(Q_{0,i,j})^{t_{ij}}.$$
Up to $\tH_i$-conjugation, we may assume that $Q_{0,i,j}=\tQ_{0,i,j}\cap H_i$, where $\tQ_{0,i,j}$ is a radical subgroup of $\tH_i$ which is a direct product of basic subgroups, in the sense of \cite{AF90,An94}.
In \cite[\S3]{FLZ20a}, a twisted version of radical subgroups of linear or unitary groups was introduced and the twisted basic subgroups are stable under the action of field and automorphisms.
So in our case, we may assume that $Q_{0,i,j}$ is $\cB_i$-stable for all $i$, $j$.
Thus 
$$\N_A(Q_0)=\prod_{i=1}^u\prod_{j=1}^{m_i}\N_{\tH_i}(Q_{0,i,j})\cB_i\wr \fS_{t_{ij}}.$$

By Proposition~\ref{lem:Ladisch}, it suffice to show that
there is a character  $\phi$ of $\N_{L_0'}(Q_0)$, which is $\prod_{i=1}^{u} \tH_i \wr\fS_{t_i}$-conjugate to $\vhi_0$, such that
$\N_A(Q_0)_\phi=\N_{\tL'}(Q_0)_\phi\N_{B}(Q_0)_\phi$ and
$\phi$ extends to
$\N_{\tL'}(Q_0)_\phi$ and $\N_{B}(Q_0)_\phi$, where
$B=\prod_{i=1}^{u} H_i\cB_i \wr\fS_{t_i}$.
Notice that it is obvious that $\phi$ extends to
$\N_{\tL'}(Q_0)_\phi$.

Now let $$\vhi_0=\prod_{i=1}^u\prod_{j=1}^{m_i}\prod_{k=1}^{l_j} (\vhi_{0,i,j,k})^{t_{ijk}}$$ where $\vhi_{0,i,j,k}\in\Irr(\N_{H_i}(Q_{0,i,j})/Q_{0,i,j})$ and $\sum_{k=1}^{l_j}t_{ijk}=t_{ij}$.
Up to $\prod_{i=1}^u\tH_i\wr \fS_{t_i}$-conjugation, we may assume that 
for every fixed $i$ and $j$,
any two of $\vhi_{0,i,j,k}$'s ($1\le k\le l_i$) are not $\N_{\tH_i}(Q_{0,i,j})$-conjugate.

Now By \cite[Thm.~7.1]{FLZ20a}, there exists $\phi_{i,j,k}\in\Irr(\N_{H_i}(Q_{0,i,j})/Q_{0,i,j})$ which is $\N_{\tH_i}(Q_{0,i,j})$-conjugate to $\vhi_{0,i,j,k}$, such that $$(\N_{\tH_i}(Q_{0,i,j})\cB_i)_{\phi_{i,j,k}}=\N_{\tH_i}(Q_{0,i,j})_{\phi_{i,j,k}}(\cB_i)_{\phi_{i,j,k}}$$ and $\phi_{i,j,k}$ extends to $\N_{H_i}(Q_{0,i,j})(\cB_i)_{\phi_{i,j,k}}$.
Let $$\phi=\prod_{i=1}^u\prod_{j=1}^{m_i}\prod_{k=1}^{l_j} (\phi_{i,j,k})^{t_{ijk}}.$$
Then $$\N_A(Q_0)_\phi=\prod_{i=1}^u\prod_{j=1}^{m_i}\prod_{k=1}^{l_j}\N_{\tH_i}(Q_{0,i,j})_{\phi_{i,j,k}}(\cB_i)_{\phi_{i,j,k}}\wr \fS_{t_{ijk}}.$$
So
$\N_A(Q_0)_\phi=\N_{\tL'}(Q_0)_\phi\N_{B}(Q_0)_\phi$.
By \cite[Lemma~25.5]{Hu98}, characters extend to wreath products.
Thus $\phi$  extends to
$\N_{B}(Q_0)_\phi$.
This completes the proof.
\end{proof}

Finally, we prove our main theorem.

\begin{proof}[Proof of Theorem~\ref{main-thm}]
As pointed in the proof of the beginning of Theorem~\ref{thm:red-isolated}, we can assume that $G=\SL_n(\eps q)$ is the universal covering of the simple group $S=\PSL_n(\eps q)$ (here $\PSL_n(-q)$ denotes $\PSU_n(q)$).	
Thanks to \cite[Thm.~C]{Sp13}, we only need to consider the non-defining characteristic $\ell\nmid q$.
By Theorem~\ref{thm:red-isolated}, it suffices to verify the inductive BAW condition for the isolated blocks.
According to \cite[Prop.~5.2]{Bo05}, the  isolated blocks of $G$ are just the unipotent blocks, which is BAW-good by \cite[Thm.~2]{FLZ20b}.
Thus we complete the proof.	
\end{proof}


\end{document}